\documentclass[,reqno]{amsart}
\usepackage{comment}
\usepackage{amsmath,amsfonts,amsthm}
\usepackage{tikz}
\newtheorem{theorem}{Theorem}[section]
\newtheorem{prop}[theorem]{Proposition}

\newtheorem{definition}[theorem]{Definition}
\newtheorem{remark}[theorem]{Remark}

\usepackage[a4paper,
top=24mm,bottom=24mm,left=26mm,right=26mm]{geometry}

\usepackage[numbers,sort]{natbib}
\usepackage[
bookmarks=true,
bookmarksdepth=subsection,
bookmarksnumbered=true,
breaklinks=true,
colorlinks=true,
linkcolor=blue,
citecolor=magenta
]{hyperref}

\usepackage[]{graphicx}
\usepackage{graphbox}

\usepackage{mathtools}
\usepackage{titlesec}
\titleformat{\section}
{\normalfont\sffamily\bfseries\scshape
}{\thesection.}{.5em}{}
\titleformat{\subsection}
{\sffamily\bfseries}{\thesubsection}{.5em}{}
\titleformat{\subsubsection}
{\sffamily
}{\thesubsubsection}{.5em}{}
\numberwithin{equation}{section}
\newcommand{\I}{\mathrm{i}}
\newcommand{\E}{\mathrm{e}}

\DeclareMathOperator{\Tr}{Tr}
\DeclareMathOperator{\Vol}{Vol}

\DeclareMathOperator{\ch}{ch}

\DeclareMathDelimiter{\Norm}{\mathord}{largesymbols}{"3E}{largesymbols}{"3E}

\DeclareMathOperator{\im}{im}
\allowdisplaybreaks
\begin{document}

\baselineskip 16pt
\parskip 8pt
\sloppy


\title[]{Torus Links $T_{2s,2t}$ and $(s,t)$-log VOA}


\author[K. Hikami]{Kazuhiro Hikami}

\address{Faculty of Mathematics,
  Kyushu University,
  Fukuoka 819-0395, Japan.}

\email{
  \texttt{khikami@gmail.com}
}

\author[S. Sugimoto]{Shoma Sugimoto}
\address{Faculty of Mathematics,
  Kyushu University,
  Fukuoka 819-0395, Japan.}

\email{
  \texttt{shomasugimoto361@gmail.com}
}



\date{\today}

\begin{abstract}
  We reveal  an intimate  connection between
  the torus link $T_{2s,2t}$ and the logarithmic $(s,t)$ VOA.
  We show that   the singlet  character of $(s,t)$-log VOA at the root
  of unity
  coincides with the Kashaev invariant and that it has a property of
  the quantum modularity.
  Also shown is that
  the tail of the
  $N$-colored Jones polynomial gives the character.
  Furthermore
  we propose a geometric method to computer the character.
\end{abstract}



\subjclass[2000]{
}


\maketitle
\section{Introduction}
Quantum invariants of knots and 3-manifolds are fascinating topics
from both physics and mathematics.
Recent studies reveal
intriguing connections with geometry,
number theory, and representation theory.

From a geometric side,
a key object is
the Kashaev invariant $\langle K \rangle_N$~\cite{Kasha95},
which
is believed to have a
structure of 
hyperbolic geometry in a large $N$ limit via the volume conjecture
\begin{equation}
  \label{eq:6}
  \lim_{N\to\infty}\frac{2\pi}{N} \log \left|\langle K \rangle_N
  \right|
  =
  \Vol(S^3\setminus K) ,
\end{equation}
where $\Vol$ denotes a hyperbolic volume.
It is well known~\cite{MuraMura99a}
that
the Kashaev invariant
$\langle K\rangle_N$ for a knot $K$
is a specific  value of the $N$-colored Jones polynomial
$J_N(q;K)$, which is a $\mathcal{U}_q(\mathfrak{sl}_2)$
knot invariant with $N$-dimensional irreducible representation;
\begin{equation}
  \label{Kashaev_Jones}
  \langle K \rangle_N
  =J_N(\zeta_N ; K),
  \qquad
  \zeta_N=\E^{\frac{2\pi\I}{N}} .
\end{equation}

Through extensive studies on the Kashaev invariant,
a notion of the quantum modular form were proposed~\cite{Zagier09a}.
A typical example of the quantum modular form is the
Kontsevich--Zagier series~\cite{DZagie01a}, which was generalized to 
those corresponding to the Kashaev invariant for the torus knot $T_{2,2t+1}$~\cite{KHikami03c}.
These results  suggest  that  the quantum invariant of knots and
3-manifolds
has an intimate
connection with a $q$-series,
which has a similar property with mock modular
forms~\cite{LawrZagi99a,KHikami03a}.

Such a $q$-series is reminiscent of the character of logarithmic
conformal field theories.
See~\cite{ChChFeFeGuHaPa22a} where  studied
was  a relationship between  the WRT invariant for 3-manifolds and the
character of VOA.
Therein the character of $(s,t)$-log VOA  explicitly given
in~\cite{FeiGaiSemTip06a} plays a role.
Later in~\cite{BringMilas15a}
the character of the singlet $(1,t)$-log VOA was
identified with a tail of the colored Jones
polynomial for  torus link $T_{2,2t}$
which was proved to exist for alternating link~\cite{DasbaXSLin06a}.
This indicates that not only the WRT invariant but the quantum knot
invariant  could have a connection with the character of VOA.

The purpose of this letter is to study a relationship between the
colored Jones polynomial for torus link $T_{2s,2t}$ and
the character of $(s,t)$-log VOA.
In Sections~\ref{sec:Jones_link} and~\ref{sec:Kashaev_link},
we introduce 
Laurent polynomials
as a family of the colored Jones polynomial
$J_N(q;T_{2s,2t})$.
In Section~\ref{sec:log-VOA},
we shall show that they coincide with the singlet characters of
$(s,t)$-log VOA at the root of unity.
We also discuss that the tail of $J_N(q;T_{2s,2t})$ also gives the character.
Section~\ref{sec:modular} is devoted to a quick review of
properties of modular forms
and their Eichler integrals.
In Section \ref{section: logVOA and Atiyah-Bott formula}, we
propose a geometrical method to calculate
the characters of irreducible modules of $(s,t)$-log VOA using
the Atiyah--Bott formula \cite{AtiyahBott}.

\section{Colored Jones Polynomials for $T_{2s,2t}$}
\label{sec:Jones_link}
We assume that $s$ and $t$ are positive coprime integers.
The torus knot $T_{s,t}$ has a braid group presentation
$\left(\sigma_1 \sigma_2\dots \sigma_{s-1}\right)^t$.
Here $\sigma_i$ denotes the generators of the Artin braid group  satisfying
the braid relations
\begin{equation*}
  \sigma_i \sigma_{i+1} \sigma_i
  =\sigma_{i+1}\sigma_i\sigma_{i+1},
  \qquad\qquad
  \sigma_i \sigma_j=\sigma_j \sigma_i,
  \quad \text{for $|i-j|>1$}.
\end{equation*}
The $N$-colored Jones polynomial for
the $0$-framing torus knot $T_{s,t}$
was given in~\cite{Mort95a} 
based on~\cite{RossJone93a} as
\begin{equation}
  \label{Jones_torus_knot}
  J_N(q;T_{s,t})
  =
  \frac{q^{\frac{1}{4}s t (1-N^2)}}{q^{\frac{N}{2}} -
    q^{-\frac{N}{2}}}
  \sum_{r=-\frac{N-1}{2}}^{\frac{N-1}{2}}
  \left(
    q^{s t r^2 - (s+t)r+\frac{1}{2}}
    - q^{ s t r^2 -(s-t) r - \frac{1}{2}}
  \right) .
\end{equation}
Here the invariant
$J_N(q;K)$
is normalized so that $J_N(q;\text{unknot})=1$.

We have interests in the 2-component
torus link $T_{2s, 2t}$, which has a
braid group
presentation  
$\left(\sigma_1\sigma_2\dots  \sigma_{2s-1}\right)^{2t}$
as in  Fig.~\ref{fig:torus_link}.
Therein
we have used  the braid relation to see that it is 
a cabling of 
the torus knot $T_{s,t}$.
As is shown in Fig.~\ref{fig:curl},
the braid $\sigma_i^2$ corresponds to the twist which is a central
element of the ribbon category.
Then,
by
replacing the braid $\sigma_i^2$ by the twists in Fig.~\ref{fig:torus_link},
the $N$-colored Jones polynomial for $T_{2s,2t}$ can be given
by use of $J_N(q;T_{s,t})$ in~\eqref{Jones_torus_knot} 
as
\begin{equation}
  \label{Jones_link_2s2t}
  J_N(q;T_{2s,2t})
  =
  \frac{
    q^{s t (1-N^2)}
  }{q^{\frac{N}{2}}-q^{-\frac{N}{2}}}
  \sum_{j=0}^{N-1}
  \sum_{k=-j}^j
  \left(
    q^{s t k^2-(s+t)k+\frac{1}{2}} -
    q^{s t k^2-(s-t)k-\frac{1}{2}}
  \right) .
\end{equation}
Here both two components of the link are assigned
the $N$-dimensional irreducible representation
of $\mathcal{U}_q(\mathfrak{sl}_2)$.
We note that
the colored HOMFLY polynomial for torus link are given in terms of the Schur function~\cite{LinZhe10a}.
\begin{figure}[htbp]
  \centering
  \includegraphics[scale=.8, align=c]{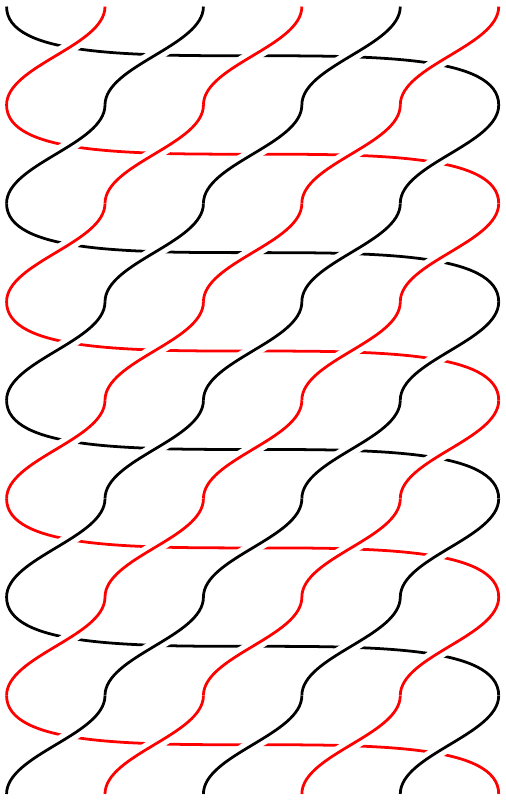}
  \quad $\approx$ \quad
  \includegraphics[scale=.8, align=c]{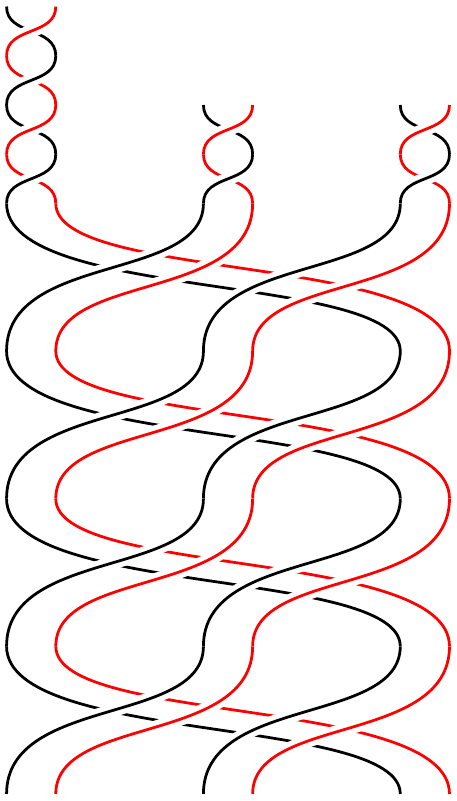}
  \caption{A braid group presentation for the
    torus link   $T_{6,8}$.
    The second component of the link is in red.
    The right hand side is
    an isotopic diagram, which shows
    that $T_{6,8}$
    is a cabling of
    the torus knot $T_{3,4}$.
  }
  \label{fig:torus_link}
\end{figure}

\begin{figure}[htbp]
  \centering
  \includegraphics[scale=.8,  align=c]{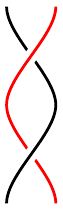}
  \quad $\approx$ \quad
  \includegraphics[scale=.8,  align=c]{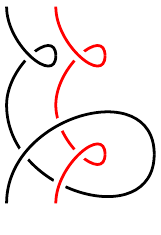}
  \caption{An isotopy of $\sigma_i^2$.}
  \label{fig:curl}
\end{figure}

One sees that,
for a sufficiently  large $N$, the tail of the colored Jones
polynomial stabilize,
and the polynomial is read as
\begin{multline}
  \label{large_N_Jones_link}
  J_N(q;T_{2s,2t})
  =
  N q^{\frac{N-1}{2}+
    s t (1-N^2)}
  \\
  \times
  \left(
    1 -  q -\frac{N-1}{N} q^{(s-1)(t-1)}
    +\frac{N-1}{N} q^{s t+s -t}+
    \frac{N-1}{N} q^{s t-s+t}
    -\frac{N-1}{N} q^{(s+1)(t+1)}
    +\dots
  \right) .
\end{multline}
We will give a proof later in~\eqref{colored Jones and (s,t)-logVOA}.

For our later use, we recall that~\cite{KHikami03a}
\begin{equation}
  \label{Jones_22p}
  J_N(q;T_{2,2p})
  =
  \frac{q^{p(1-N^2)}}{q^{\frac{N}{2}}-q^{-\frac{N}{2}}}
  \sum_{j=0}^{N-1}
  q^{p j(j+1)}
  \left(
    q^{j+\frac{1}{2}}-q^{-j-\frac{1}{2}}
  \right)
\end{equation}

\section{Modular Forms and Eichler Integrals}
\label{sec:modular}
\subsection{Unary Theta Series}
We introduce periodic functions with mean values zero as follows. 
\begin{gather}
  \label{eq:10}
  \psi_{2p}^{(a)}(k)=
  \begin{cases}
    \pm 1, & \text{for $k=\pm a \mod 2p$},
    \\
    0, & \text{otherwise},
  \end{cases}
  \\
  \chi_{2st}^{(n,m)}(k)
  =
  \begin{cases}
    1, & \text{for $k=\pm (n t - ms ) \mod 2 s t $,}\\
    -1, & \text{for $k=\pm (n t + m s) \mod 2 s t$,} \\
    0, & \text{otherwise.}
  \end{cases}
\end{gather}
Here $s$ and $t$ are coprime positive integers.
We assume that
$0<a<p$,
and $0<n<s$, $0<m<t$.
See that $\chi_{2st}^{(n,m)}(k)=\chi_{2st}^{(s-n,t-m)}(k)$.
The unary theta series  are defined by
\begin{gather}
  \label{define_Psi}
  \Psi_p^{(a)}(\tau)
  =\frac{1}{2}\sum_{k\in \mathbb{Z}} k \, \psi_{2p}^{(a)}(k) \,
  q^{\frac{k^2}{4p}} ,
  \\
  \label{define_Phi}
  \Phi_{s,t}^{(n,m)}(\tau)
  =
  \frac{1}{2}
  \sum_{k\in\mathbb{Z}}
  \chi_{2st}^{(n,m)}(k) \, q^{\frac{k^2}{4st}} ,
\end{gather}
where $q=\E^{2\pi\I \tau}$ for
$\tau \in \mathbb{H}$.
The $q$-series, $\Psi_p^{(a)}(\tau)$ and
$\Phi_{s,t}^{(n,m)}(\tau)$, are vector-valued modular forms with
weight $\frac{3}{2}$ and $\frac{1}{2}$, respectively.
We have
\begin{gather}
  \label{eq:15}
  \begin{aligned}[t]
    & \Psi_p^{(a)}(\tau)
      =
      \left(\frac{\I}{\tau}\right)^{\frac{3}{2}}
      \sum_{b=1}^p \sqrt{\frac{2}{p}}
      \sin\left( \frac{a b}{p}\pi \right) \,
      \Psi_p^{(b)}\left( - \tfrac{1}{\tau}\right) ,
    \\
    &
      \Psi_p^{(a)}(\tau+1)=\E^{\frac{a^2}{2p}\pi\I}
      \,
      \Psi_p^{(a)}(\tau) ,
  \end{aligned}
  \\
  \begin{aligned}[t]
    &
      \Phi_{s,t}^{(n,m)}(\tau)
      =
      \sqrt{\frac{\I}{\tau}}
      \sideset{}{'}\sum_{n^\prime,m^\prime}
      S(s,t)_{n,m}^{n^\prime,   m^\prime}
      \Phi_{s,t}^{(n^\prime, m^\prime)}\left(-\tfrac{1}{\tau}\right) ,
    \\
    &
      \Phi_{s,t}^{(n,m)}(\tau+1)=
      \E^{\frac{(n t - m s)^2}{2 s t}\pi\I}
      \Phi_{s,t}^{(n,m)}(\tau) ,
  \end{aligned}
\end{gather}
where
$\Sigma^\prime_{n^\prime,m^\prime}$
means that
$n^\prime$ and $m^\prime$ runs over
a $\frac{1}{2}(s-1)(t-1)$-dimensional space,
and
\begin{equation}
  S(s,t)_{n,m}^{n^\prime,m^\prime}
  =\sqrt{\frac{8}{s t}}(-1)^{n m^\prime+n^\prime m+1}
      \sin\left(n n^\prime \frac{t}{s}\pi\right)
      \sin\left(m m^\prime \frac{s}{t}\pi\right) .
\end{equation}
The weight~$0$ modular forms,
$\frac{\Psi_P^{(a)}(\tau)}{[\eta(\tau)]^3}$
and
$\frac{\Phi_{s,t}^{(n,m)}(\tau)}{\eta(\tau)}$
where $\eta(\tau)$ denotes the Dedekind $\eta$-function,
are  characters of the
$A^{(1)}_1$ conformal field theory and
the Virasoro algebra
$\operatorname{Vir}_{(s,t)}
=\mathcal{W}_2(s,t)
$ 
respectively.

\subsection{Eichler Integrals}
Following~\cite{LawrZagi99a,DZagie01a,KHikami03c,KHikami03a},
we introduce the Eichler integrals of the vector modular forms~\eqref{define_Psi} and~\eqref{define_Phi}
as
\begin{gather}
  \label{tilde_Psi}
  \widetilde{\Psi}_p^{(a)}(\tau)
  =
  \sum_{k=0}^\infty \psi_{2p}^{(a)}(k) \, q^{\frac{k^2}{4p}} ,
  \\
  \label{tilde_Phi}
  \widetilde{\Phi}_{s,t}^{(n,m)}(\tau)
  =
  -\frac{1}{2}\sum_{k=0}^\infty k \, \chi_{2 s t}^{(n,m)}(k) \,
  q^{\frac{k^2}{4 s t}} .
\end{gather}
Limiting values
when $\tau \downarrow \frac{1}{N}$ for $N\in\mathbb{Z}_{>0}$
were given in~\cite{KHikami03a,KHikami03c}
as
\begin{gather}
  \label{limit_Eichler_Psi}
  \widetilde{\Psi}_{p}^{(a)}\left(\tfrac{1}{N}\right)
  =
  -\sum_{k=1}^{2p N}
  \psi_{2p}^{(a)}(k) \, \E^{\frac{k^2}{2p N}\pi \I} \,
  B_1\left(\tfrac{k}{2 p N}\right) ,
  \\
  \label{limit_Eichler_Phi}
  \widetilde{\Phi}_{s,t}^{(n,m)}
  \left(\tfrac{1}{N}\right)
  =
  \frac{s t N}{2}
  \sum_{k=1}^{2s t N}
  \chi_{2st}^{(n,m)}(k) \, \E^{\frac{k^2}{2s t N}\pi\I} \,
  B_2\left(
    \tfrac{k}{2 s t N}
  \right) ,
\end{gather}
where $B_n(x)$ is the $n$-th Bernoulli polynomials,
$B_1(x)=x-\frac{1}{2}$ and
$B_2(x)=x^2-x+\frac{1}{6}$.
We note that
these limiting values were closely related  with the Kashaev invariant
$\langle T_{2,2p} \rangle_N$ and
$\langle T_{s,t} \rangle_N$ respectively~\cite{KHikami03a,KHikami03c};
\begin{align}
  \label{eq:2}
  \left\langle T_{2,2p} \right\rangle_N
  & =
    - p N
    \zeta_N^{~\frac{3p^2-1}{4p}} \,
    \widetilde{\Psi}_{p}^{(p-1)}
    \left( \tfrac{1}{N} \right) ,
  \\
  \left\langle T_{s,t} \right\rangle_N
  & =
    \zeta_N^{~\frac{s^2t^2-s^2-t^2}{4st}} \,
    \widetilde{\Phi}_{s,t}^{(s-1,1)}
    \left( \tfrac{1}{N} \right) .
\end{align}
These follow from~\eqref{Jones_torus_knot} and~\eqref{Jones_22p}.
See also~\cite{KHikami04b,KHikami05a,KHikami10a} for a relationship
with the WRT invariant
for  Seifert manifolds.

Asymptotic expansions
of 
the Kashaev invariants
$\langle T_{2,2p}\rangle_N$ and
$\langle T_{s,t} \rangle_N$
in $N\to\infty$
follow from~\cite{KHikami03a,KHikami03c}
\begin{gather}
  \label{modular_tilde_Psi}
  \widetilde{\Psi}_p^{(a)}\left(\tfrac{1}{N}\right)
  +
  \sqrt{\frac{N}{\I}}
  \sum_{b=1}^{p-1} \sqrt{\frac{2}{p}}
  \sin\left(\frac{a b}{p}\pi\right) \,
  \left(1-\frac{b}{p}\right) \E^{-\frac{b^2}{2p}\pi \I N}
  \simeq
  \sum_{k=0}^\infty
  \frac{L\left( -2k, \psi_{2p}^{(a)} \right)}{k!}
  \left(
    \frac{\pi \I}{2p N}
  \right)^k ,
  \\
  \label{modular_tilde_Phi}
  \widetilde{\Phi}_{s,t}^{(n,m)}\left(\tfrac{1}{N}\right)
  +\left( \frac{N}{\I} \right)^{\frac{3}{2}}
  \sideset{}{'}\sum_{n^\prime,m^\prime}
  S(s,t)_{n,m}^{n^\prime,m^\prime}
  \phi_{s,t}(n^\prime,m^\prime)\,
  \E^{-\frac{(n^\prime t - m^\prime s)^2}{2st}\pi \I N}
  \simeq
  \frac{-1}{2}
  \sum_{k=0}^\infty
  \frac{L\left(-2k-1,\chi_{2st}^{(n,m)}\right)}{
    k!}
  \left(
    \frac{\pi \I}{2 s t N}
  \right)^k ,
\end{gather}
where
\begin{equation}
  \phi_{s,t}(n,m)=
  \begin{cases}
    (s-n)m , & n t > m s ,
    \\
    n(t-m), & n t< m s .
  \end{cases}
\end{equation}
These prove  the quantum modularity~\cite{Zagier09a} of the Eichler
integrals~\eqref{tilde_Psi} and~\eqref{tilde_Phi}.


\section{Kashaev Invariant of $T_{2s,2t}$}
\label{sec:Kashaev_link}

As a family of the $N$-colored Jones
polynomial~\eqref{Jones_link_2s2t}
for
$T_{2s,2t}$,
we define
\begin{equation}
  \label{general_J}
  {\mathcal{J}_N}(q;
  \begin{smallmatrix}
    (n,m)\\
    2s,2t
  \end{smallmatrix}
  )
  =
  \frac{1}{q^{\frac{N}{2}}-q^{-\frac{N}{2}}}
  \sum_{c=0}^{N-1} \sum_{r=-c}^c
  \left(
    q^{s t r^2-(n t+m s) r + \frac{m n}{2}}
    - q^{s t r^2+(n t-m s)r-\frac{m n}{2}}
  \right) .
\end{equation}
We have
${\mathcal{J}_N}
(q;
\begin{smallmatrix}
  (n,m) \\
  2s,2t
\end{smallmatrix}
)\in \mathbb{Z}[q^{\frac{1}{2}},q^{-\frac{1}{2}}]$ due to
\begin{equation}
  \label{ex_Gauss_sum}
  \sum_{k=0}^{N-1}
  \left(
    \zeta_N^{~ \frac{(2 s t k - (n t + m s))^2}{4 s t}}
    -
    \zeta_N^{~ \frac{\left( 2 s t k + (n t - m s) \right)^2}{4 s t}}
  \right)=0 .
\end{equation}
One sees that
\begin{equation}
  {\mathcal{J}_N}
  (q;
  \begin{smallmatrix}
    (1,1)\\
    2s, 2t
  \end{smallmatrix}
  )=
  q^{-s t (1-N^2)}
  J_N(q;T_{2s, 2t}) ,
\end{equation}
and that
the Kashaev invariant~\eqref{Kashaev_Jones} for $T_{2s,2t}$ is given as
\begin{equation}
  \label{eq:1}
  \langle T_{2s,2t}\rangle_N
  =
  J_N(\zeta_N ; T_{2s,2t})
  =
  \zeta_N^{~  s t}
  {\mathcal{J}_N}
  (\zeta_N;
  \begin{smallmatrix}
    (1,1)\\
    2s,2t
  \end{smallmatrix}
  ) .
\end{equation}

We shall confirm the quantum modularity of
the Kashaev invariant $\langle T_{2s,2t}\rangle_N$.
At the $N$-th root of unity~$\zeta_N$,
the Laurent polynomial~\eqref{general_J}
reduces to
\begin{equation}
  \begin{aligned}[b]
    {\mathcal{J}_N}(\zeta_N ;
    \begin{smallmatrix}
      (n,m)\\
      2s, 2t
    \end{smallmatrix}
    )
    & =
      \frac{1}{N}
      \sum_{c=0}^{N-1}\sum_{r=-c}^c f(r)
    \\
    &
      =
      f(0) +
      \frac{1}{N}\sum_{k=1}^{N-1}
      \left\{ (N-k) \, f(k)+k \,  f(k-N)\right\} ,
  \end{aligned}
\end{equation}
where for brevity we mean
\begin{equation*}
  f(r) =
  \left(
    s t r^2-(n t+m s) r + \frac{m n}{2}
  \right)
  \zeta_N^{~ s t r^2-(n t+m s) r + \frac{m n}{2}}
  -
  \left(
    s t r^2+(n t-m s)r-\frac{m n}{2}
  \right)
  \zeta_N^{~ s t r^2+(n t-m s)r-\frac{m n}{2}} .
\end{equation*}
Then we get
\begin{equation*}
  \begin{aligned}
    &
    \zeta_N^{~ \frac{ (n t)^2+(m s)^2}{4 s t}}
      {\mathcal{J}_N}(\zeta_N;
      \begin{smallmatrix}
        (n, m)\\
        2s, 2t
      \end{smallmatrix}
      )
    \\
    & =
      \sum_{k=0}^{N-1}
      \left\{
      \left( s t k(N-k)+\frac{m n}{2}\right) \,
      \zeta_N^{~ \frac{(2 s t k - (n t + m s))^2}{4 s t}}
      -
      \left( s t k(N-k)-\frac{m n}{2}\right) \,
      \zeta_N^{~ \frac{(2 s t k + (n t - m s))^2}{4 s t}}
      \right\}
    \\
    & =
      \sum_{k=0}^{N-1}
      \Bigl\{
      \left(
      - s t N^2 B_2\left(
      \tfrac{2 s t k - (n t + m s)}{2 s t N}
      \right)
      - (n t + m s) N \, B_1\left(
      \tfrac{2 s t k - (n t + m s)}{2 s t N}
      \right)
      \right) \,
      \zeta_N^{~ \frac{(2 s t k - (n t + m s))^2}{4 s t}}
    \\
    & \qquad \qquad
      -
      \left(
      - s t N^2 B_2\left(
      \tfrac{2 s t k + (n t - m s)}{2 s t N}
      \right)
      + (n t - m s) N \, B_1\left(
      \tfrac{2 s t k + (n t - m s)}{2 s t N}
      \right)
      \right) \,
      \zeta_N^{~ \frac{\left( 2 s t k + (n t - m s) \right)^2}{4 s t}}
    \Bigr\} ,
  \end{aligned}
\end{equation*}
where  we have used~\eqref{ex_Gauss_sum} in the last equality.
Recalling~\eqref{limit_Eichler_Psi} and~\eqref{limit_Eichler_Phi},
we conclude that the
${\mathcal{J}_N}
(\zeta_N;
\begin{smallmatrix}
  (n,m)\\
  2s, 2t
\end{smallmatrix}
)$
can be written as a sum of limiting values of
the Eichler integrals;
\begin{equation}
  \label{J_and_Eichler}
  \frac{1}{N} \,
  \zeta_N^{~ \frac{ (n t)^2+(m s)^2}{4 s t}}
  {\mathcal{J}_N}
  (\zeta_N;
  \begin{smallmatrix}
    (n,m)\\
    2s , 2t
  \end{smallmatrix}
  )
  =
  - \widetilde{\Phi}_{s,t}^{(n,m)}\left(\tfrac{1}{N}\right)
  -\frac{n t - m s}{2} \,
  \widetilde{\Psi}_{s t }^{(n t - m s)}\left(\tfrac{1}{N}\right)
  +\frac{n t + m s}{2} \,
  \widetilde{\Psi}_{s t }^{(n t + m s)}\left(\tfrac{1}{N}\right) .
\end{equation}
As a  consequence of
~\eqref{modular_tilde_Psi}
and~\eqref{modular_tilde_Phi},
we obtain
the quantum modularity of
$\mathcal{J}_N(q;
\begin{smallmatrix}
  (n,m)\\
  2s, 2t
\end{smallmatrix}
)$.

\section{log VOA}
\label{sec:log-VOA}
In this section, we consider the case of $\mathfrak{sl}_2$ of the lattice VOA.
We denote $\alpha$ and $\varpi$ by the simple root and the fundamental weight, respectively.
For a VOA or its module $M$, $\ch_q$ means
$\Tr_Mq^{L_0-\frac{c}{24}}$,
and
$\ch_{q,z}=\Tr_M q^{L_0 - \frac{c}{24}} z^h$.

\subsection{$(s,t)$-log VOA}
Let us consider the lattice VOA $V_{\sqrt{st}Q}$ associated with the rescaled root lattice $\sqrt{st}Q=\sqrt{2st}\mathbb{Z}$.
The irreducible modules of $V_{\sqrt{st}Q}$ are given by $V_{n,m}^+=V_{\sqrt{st}Q+\alpha_{n,m}}$ and $V_{n,m}^-=V_{\sqrt{st}(Q-\varpi)+\alpha_{n,m}}$, where for $1\leq n\leq s$ and $1\leq m\leq t$, set
\begin{align}
  \alpha_{n,m}:=\frac{-t(n-1)+s(m-1)}{\sqrt{st}}\varpi,
  \qquad
  \Delta_{n,m,k}:=\frac{(ms-nt+stk)^2}{4st}.
  \label{define_Delta_n,m,k}
\end{align}
We note that
\begin{align}\label{relation among conformal weight}
  \Delta_{n,m,k}=\Delta_{-n,-m,-k}=\Delta_{s+n,t+m,k}. 
\end{align}
Let $\mathcal{L}$ be the Virasoro algebra at the central charge $c=1-6\frac{(s-t)^2}{st}$.
The Virasoro VOA $\operatorname{Vir}_{s,t}=U(\mathcal{L})|0\rangle$ is a sub VOA of the Heisenberg VOA $V_{\sqrt{st}Q}^{h=0}$.
Then the conformal weight of $e^{\sqrt{st}k\varpi+\alpha_{n,m}}$ is $\Delta_{n,m,k}+\frac{c}{24}$ for the central charge $c=1-6\frac{(s-t)^2}{st}$.
In particular, we have
\begin{align}
  \ch_qV_{n,m}^+=\sum_{k\in\mathbb{Z}}\frac{q^{\Delta_{n,m,2k}}}{\eta(\tau)},
  \qquad
  \ch_qV_{n,m}^-=\sum_{k\in\mathbb{Z}}\frac{q^{\Delta_{n,m,2k+1}}}{\eta(\tau)}.
\end{align}
To define the $(s,t)$-log VOA and its irreducible module, we need the short screening operators 
\begin{align}
  \mathcal{Q}^{[n]}_{+}\colon V_{n,m}^{\pm}\rightarrow V_{s-n,m}^{\pm},
  \qquad
  \mathcal{Q}^{[m]}_{-}\colon V_{n,m}^{\pm}\rightarrow V_{n,t-m}^{\pm}
\end{align}
in \cite[Definition 3.23]{TsuchiyaWood}.
Then the $(s,t)$-log VOA $\mathcal{K}_{1,1}^+$ is defined by $\mathcal{K}_{1,1}^+=\ker \mathcal{Q}^{[1]}_{+}\cap\ker \mathcal{Q}^{[1]}_{-}$~\cite{FeiGaiSemTip06a, TsuchiyaWood}.

\subsection{Characters  and Kashaev Invariant}
\label{subsec: characters of logVOA}
It is known that
there are $2st +\frac{1}{2}(s-1)(t-1)$ irreducible modules of
the $(s,t)$-log VOA $\mathcal{K}_{1,1}^+$.
The characters of irreducible modules
$\mathcal{X}_{n,m}^{\pm}
= \im \mathcal{Q}^{[s-n]}_{+}\cap\im \mathcal{Q}^{[t-m]}_{-}$
were computed
in~\cite{FeiGaiSemTip06a}
as
\begin{align}
  \label{character_X+}
  \ch_q\mathcal{X}_{n,m}^{+}
& =
  \frac{1}{\eta(\tau)}
  \sum_{k\in \mathbb{Z}}
  k^2
  \left(
  q^{\frac{(2s t k - n t - m s)^2}{4 s t}}
  -
  q^{\frac{(2 s t k- n t+m s)^2}{4 s t}}
  \right) ,
  \\
  \label{character_X-}
  \ch_q\mathcal{X}_{n,m}^{-}
& =
  \frac{1}{\eta(\tau)}
  \sum_{k\in \mathbb{Z}}
  k(k+1)
  \left(
  q^{\frac{(2s t k+s t - n t - m s)^2}{4 s t}}
  -
  q^{\frac{(2 s t k+ s t - n t+m s)^2}{4 s t}}
  \right) .
\end{align}
Amongst others, the singlet character was explicitly written
in~\cite{ChChFeFeGuHaPa22a} as 
\begin{equation}
  \label{singlet_character}
  \begin{aligned}[b]
    \eta(\tau) \,
    \ch_q(\mathcal{X}_{n,m}^{+})^{h=0}
    & =
      \sum_{k\in \mathbb{Z}}
      |k| \left(
      q^{\frac{(2s t k - n t - m s)^2}{4 s t}}
      -
      q^{\frac{(2 s t k- n t+m s)^2}{4 s t}}
      \right)
    \\
    & =
      \frac{1}{s t}
      \left(
      \widetilde{\Phi}_{s,t}^{(n,m)}\left(\tau\right)
      +\frac{n t - m s}{2}
      \widetilde{\Psi}_{st}^{(n t - m s)}\left(\tau\right)
      -\frac{n t + m s}{2}
      \widetilde{\Psi}_{st}^{(n t + m s)}\left(\tau\right)
      \right) .
  \end{aligned}
\end{equation}

We point out that,
in view of~\eqref{J_and_Eichler},
the Laurent polynomial~\eqref{general_J} at the $N$-th root of unity
coincides with a limiting value of the singlet
character~\eqref{singlet_character} up to multiples.
\begin{theorem}
  The Kashaev invariant $\langle T_{2s,2t} \rangle_N$  is a limiting
  value of the character of $(s,t)$-log VOA
  $\ch_q(\mathcal{X}_{1,1}^+)^{h=0}$
  (up to the Dedekind $\eta$-function).  
\end{theorem}
Note that
in~\cite{ChChFeFeGuHaPa22a}
discussed was a relationship between the singlet character and 
the WRT invariants for
4-fibered Seifert manifolds.
See~\cite{KHikami04e} for quantum modularity.
See also \cite{FujiIwakMuraTera08a,MatsuTeras21a}.

\subsection{Characters and Tail of the Colored Jones Polynomial}

The relationship between the singlet
characters~\eqref{singlet_character}
and
the Laurent polynomial~\eqref{general_J}
can also be seen in different manner.
As a generalization of~\eqref{large_N_Jones_link},
we have the following correspondence.
\begin{theorem}\label{main theorem torus link and logVOA}
  The tail of the $N$-colored Jones polynomial coincides with
  the characters of  the
  $(s,t)$-log VOA;
\begin{equation}\label{colored Jones and (s,t)-logVOA}
  \lim_{N\to\infty}
  \left(
    q^{\frac{( n t)^2+(m s)^2}{4 s t}
      -\frac{N}{2}}
    {\mathcal{J}_N}
    (q;
    \begin{smallmatrix}
      (n,m) \\
      2s, 2t
    \end{smallmatrix}
    )
    -
    N \,\Phi_{s,t}^{(n,m)}(\tau)
  \right)
  =
  \eta(\tau) \,
 \ch_q(\mathcal{X}_{n,m}^{+})^{h=0}.
\end{equation}
\end{theorem}
\begin{proof}
This can be proved as follows.
We have
\begin{equation*}
\begin{aligned}[b]
  &
    \left(q^{\frac{N}{2}}-q^{-\frac{N}{2}}\right) \,
    {\mathcal{J}_N}(q;
    \begin{smallmatrix}
      (n,m)\\
      2s,2t
    \end{smallmatrix}
    )
    \\
  &=
    q^{-\frac{m^2s^2+n^2t^2}{4st}}
    \sum_{c=0}^{N-1}\sum_{r=-c}^{c}
    \left( q^{\Delta_{s-n,m,-2r+1}}-q^{\Delta_{n,m,-2r}}
    \right)
  \\
  &=
    q^{-\frac{m^2s^2+n^2t^2}{4st}}
    \left(
    N(q^{\Delta_{s-n,m,1}}-q^{\Delta_{n,m,0}})
    \vphantom{    +\sum_{r=1}^{N-1}(N-r)}
    \right.
  \\
  & \qquad \qquad \left.
    +\sum_{r=1}^{N-1}(N-r)
    (q^{\Delta_{s-n,m,2r+1}}
    -q^{\Delta_{n,m,2r}}
    +q^{\Delta_{s-n,m,-2r+1}}
    -q^{\Delta_{n,m,-2r}})
    \right)\\
  &=
  q^{-\frac{m^2s^2+n^2t^2}{4st}}
    \left(
    N(q^{\Delta_{s-n,m,1}}-q^{\Delta_{n,m,0}})
    \vphantom{+\sum_{r=1}^{N-1}(N-r)}
    \right.
  \\
  & \qquad \qquad \left.
    +\sum_{r=1}^{N-1}(N-r)
  (q^{\Delta_{s-n,m,-2r+1}}
  -q^{\Delta_{s-n,t-m,-2r}}
  -q^{\Delta_{n,m,-2r}}
    +q^{\Delta_{n,t-m,-2r-1}})
    \right)\\
  &=
    q^{-\frac{m^2s^2+n^2t^2}{4st}}
    \left(
    N(q^{\Delta_{s-n,m,1}}-q^{\Delta_{n,m,0}})
    +(q;q)_{\infty}\sum_{r=1}^{N-1}\frac{(N-r)}{r}\ch_q\mathcal{J}_{n,t-m,2r-1}
    \right)
    \\
  &=N\left(q^{\frac{mn}{2}}-q^{-\frac{mn}{2}}\right)
  +
    q^{-\frac{m^2s^2+n^2t^2}{4st}} \, \eta(\tau)
    \sum_{r=1}^{N-1}\frac{(N-r)}{r}\ch_q\mathcal{J}_{n,t-m,2r-1},
\end{aligned}
\end{equation*}
where the third equality follows
from~\eqref{relation among conformal weight},
and 
\begin{align}
\ch_q\mathcal{J}_{n,t-m,2k-1}=\frac{1}{\eta(\tau)}k(q^{\Delta_{s-n,m,-2k+1}}-q^{\Delta_{s-n,t-m,-2k}}-q^{\Delta_{n,m,-2k}}+q^{\Delta_{n,t-m,-2k-1}})
\end{align}
is the character of irreducible $L(c_{s,t},0)$-module $\mathcal{J}_{n,t-m,2k-1}$ generated by $e^{-(k-1)\sqrt{st}\alpha+\alpha_{n,m}}$ (see \cite[(3.42)]{ChChFeFeGuHaPa22a}).
Then we have
\begin{equation*}
  \begin{aligned}[b]
    &
      q^{\frac{m^2s^2+n^2t^2}{4st}-\frac{N}{2}}
      \left( q^N-1 \right) \,
      {\mathcal{J}_N}(q;
  \begin{smallmatrix}
    (n,m)\\
    2s,2t
  \end{smallmatrix}
  )\\
  =&
  Nq^{\Delta_{n,m,0}}\left(q^{mn}-1 \right)
  +
     \eta(\tau)\sum_{r=1}^{N-1}\frac{(N-r)}{r}\ch_q\mathcal{J}_{n,t-m,2r-1}\\
    =
    &N
     \left(
     q^{\Delta_{n,m,0}}(q^{mn}-1)+\eta(\tau)\sum_{r=1}^N\frac{1}{r}\ch_q\mathcal{J}_{n,t-m,2r-1}
     \right)
      -\eta(\tau)
      \sum_{r=1}^N\ch_q\mathcal{J}_{n,t-m,2r-1}\\
  =&N
     \left(
     q^{\Delta_{n,m,0}}(q^{mn}-1)
     +\sum_{r=1}^N(q^{\Delta_{s-n,m,-2r+1}}-q^{\Delta_{s-n,t-m,-2r}}-q^{\Delta_{n,m,-2r}}+q^{\Delta_{n,t-m,-2r-1}})
     \right)
    \\
    & \qquad 
      -\eta(\tau)\sum_{r=1}^N\ch_q\mathcal{J}_{n,t-m,2r-1}
\end{aligned}
\end{equation*}
and thus we get
\begin{multline*}
  \lim_{N\rightarrow\infty}
  \left(
    q^{\frac{m^2s^2+n^2t^2}{4st}-\frac{N}{2}}
    {\mathcal{J}_N}(q;
    \begin{smallmatrix}
      (n,m)\\
      2s,2t
    \end{smallmatrix}
    )
    -N\Phi^{(n,m)}_{s,t}(\tau)
  \right)
  \\
  =
  \lim_{N\rightarrow\infty}
  N
  \left(
    -\sum_{r=1}^N
    \left(
      q^{\Delta_{s-n,m,-2r+1}}-q^{\Delta_{s-n,t-m,-2r}}-q^{\Delta_{n,m,-2r}}+q^{\Delta_{n,t-m,-2r-1}}
    \right)
  \right.
  \\
  \left. \vphantom{\sum_{r=1}^N}
    -q^{\Delta_{n,m,0}}(q^{mn}-1)-\Phi^{(n,m)}_{s,t}(\tau)
    \right)
  +\eta(\tau)\ch_q(\mathcal{X}_{n,m}^{+})^{h=0}.
\end{multline*}
Because
\begin{align*}
\Phi^{(n,m)}_{s,t}(\tau)=-\sum_{r\geq
  1}
  \left(
  q^{\Delta_{s-n,m,-2r+1}}-q^{\Delta_{s-n,t-m,-2r}}-q^{\Delta_{n,m,-2r}}
  +q^{\Delta_{n,t-m,-2r-1}}
  \right)-q^{\Delta_{n,m,0}}(1-q^{mn})
\end{align*}
we have
\begin{multline*}
  \lim_{N\rightarrow\infty}
  N
  \left(
    -\sum_{r=1}^N
    \left(
      q^{\Delta_{s-n,m,-2r+1}}-q^{\Delta_{s-n,t-m,-2r}}-q^{\Delta_{n,m,-2r}}
      +q^{\Delta_{n,t-m,-2r-1}}
    \right)
  \right.
  \\
  \left.\vphantom{\sum_{r=1}^N}
  -q^{\Delta_{n,m,0}}(q^{mn}-1)-\Phi^{(n,m)}_{s,t}(\tau)
  \right)=0.
\end{multline*}
and \eqref{colored Jones and (s,t)-logVOA} is proved.
\end{proof}

Our result~\eqref{colored Jones and (s,t)-logVOA}
is motivated by~\cite{BringMilas15a}
where discussed was a relationship between
the tail of the colored Jones polynomial for the torus link
$T_{2,2t}$ and the singlet~$(1,t)$-log VOA.
We also note that
in~\cite{Kanade23a,Kanade23b}
the tail of the colored $\mathfrak{sl}_r$ polynomial for $T_{s,t}$ gives the
character of $\mathcal{W}_r(s,t)$.

\begin{remark}
We should note that the characters~\eqref{character_X+}
and~\eqref{character_X-} also appear as a tail
when we consider a 3-component torus link $T_{3s,3t}$.
We can apply
the same method with
Section~\ref{sec:Jones_link} to  obtain
\begin{equation}
  \label{triple_link}
  J_N(q;T_{3s,3t})
  =
  \frac{q^{\frac{9}{4}s t (1-N^2)}}{q^{\frac{N}{2}}-q^{-\frac{N}{2}}}
  \sum_{b=0}^{N-1}
  \sum_{\substack{
      |2b-N+1|\leq c\leq 2b+N-1
      \\
      c+N:\text{odd}
    }}
  \sum_{r=-\frac{c}{2}}^{\frac{c}{2}}
  \left(
    q^{s t r^2 - (s+t ) r+\frac{1}{2}} -
    q^{s t r^2 - (s-t ) r-\frac{1}{2}} 
  \right) .
\end{equation}
When we define a family of Laurent polynomials by
\begin{equation}
  \mathcal{J}_N(q;
  \begin{smallmatrix}
    (n,m)\\
    3s,3t
  \end{smallmatrix}
  )
  =
  \frac{1}{q^{\frac{N}{2}}-q^{-\frac{N}{2}}}
  \sum_{b=0}^{N-1}
  \sum_{\substack{
      |2b-N+1|\leq c\leq 2b+N-1
      \\
      c+N:\text{odd}
    }}
  \sum_{r=-\frac{c}{2}}^{\frac{c}{2}}
  \left(
    q^{s t r^2 - (m s +n t ) r+\frac{m n}{2}} -
    q^{s t r^2 - (m s-n t ) r-\frac{m n}{2}} 
  \right),
\end{equation}
we get~\eqref{character_X+}
and~\eqref{character_X-} by the similar computations
\begin{itemize}
\item for odd $N$
  \begin{equation}
  \lim_{N\to\infty}
  \left(
    q^{\frac{(m s)^2+(n t)^2}{4 s t}
      -\frac{N}{2}
    }
    \mathcal{J}_{N}(q;
    \begin{smallmatrix}
      (n,m)
      \\
      3s, 3t
    \end{smallmatrix}
    )
    -\frac{3N^2+1}{4}
    \Phi_{s,t}^{(n,m)}(\tau)
  \right)
  =
  \eta(\tau) \ch_q \mathcal{X}_{n,m}^+,
\end{equation}

\item for even $N$
  \begin{equation}
    \label{eq:3}
    \lim_{N\to\infty}
    \left(
      q^{\frac{(m s)^2+(n t)^2}{4 s t}
        -\frac{N}{2}
      }
      \mathcal{J}_{N}(q;
      \begin{smallmatrix}
        (n,m)
        \\
        3s, 3t
      \end{smallmatrix}
      )
      +
      \frac{3N^2}{4}
      \Phi_{s,t}^{(s-n,m)}(\tau)
    \right)
    =
    \eta(\tau) \,
    \ch_q \mathcal{X}_{n,m}^- .
  \end{equation}
\end{itemize}

\end{remark}

\section{log VOA and Atiyah--Bott formula}
\label{section: logVOA and Atiyah-Bott formula}
We explain a
geometrical
method to calculate the
character of the irreducible modules $\mathcal{X}_{n,m}^{\pm}$ of
$\mathcal{K}_{1,1}^{+}$ using the Atiyah--Bott formula.

\subsection{Geometric construction of $(1,t)$-log VOA and Atiyah--Bott formula}
The geometric construction of $(s,t)$-log VOA for $s=1$ was proposed in \cite{FeiTip10} and given a rigorous mathematical proof in \cite{Sugimoto1,Sugimoto2}.
That is, the $(1,t)$-log VOA
\footnote{In other literature, it is often represented by the symbol $W(t)_Q$ or $W_{\sqrt{t}Q}$.}
is given by the space of global sections 
\begin{align}
H^0(G\times_{B}V_{\sqrt{t}Q})
\end{align}
of the homogeneous vector bundle $G\times_{B}V_{\sqrt{t}Q}$ over the
flag variety $G/{B}$, where $Q$ is the root lattice of $G$ and
$V_{\sqrt{t}Q}$ is the lattice VOA associated with the rescaled root
lattice $\sqrt{t}Q$.
$B$ is the (lower) Borel subgroup of $G$.
Furthermore, for an irreducible module $V_{\sqrt{t}Q+\lambda}$ over $V_{\sqrt{t}Q}$, $H^0(G\times_{B}V_{\sqrt{t}Q+\lambda})$ is an $H^0(G\times_{B}V_{\sqrt{t}Q})$-module (where $\lambda=-\sqrt{t}\lambda_0+\lambda_t$ and $\lambda_0$ is a minuscule weight).
By the main results of~\cite{Sugimoto1,Sugimoto2}, for $\lambda$ such
that $(\sqrt{t}\lambda_t+\rho,\theta)\leq t$
where $\rho$ and $\theta$ are  respectively the Weyl vector and highest root,
$H^0(G\times_{B}V_{\sqrt{t}Q+\lambda})$ is irreducible as
$H^0(G\times_{B}V_{\sqrt{t}Q})$-module and
$H^k(G\times_{B}V_{\sqrt{t}Q+\lambda})=0$ for $k>0$.
In particular, by using the Atiyah--Bott fixed point formula \cite{AtiyahBott}
\begin{align}\label{Atiyah-Bott formula}
\sum_{k\geq 0}(-1)^k\ch_{q,z}H^k(G\times_{B}V)
=\sum_{\beta\in P_+}\ch_{z}L(\beta)\sum_{\sigma\in W}(-1)^{l(\sigma)}\ch_{q}V^{h=\sigma\circ\beta},
\end{align}
where $\ch_{z}L(\beta)$ is the Weyl character formula of the
irreducible module $L(\beta)$ of $\mathfrak{g}$ with highest
weight~$\beta$, we obtain the character formula
\begin{align}
\ch_{q,z}H^0(G\times_{B}V_{\sqrt{t}Q+\lambda})
=
\sum_{\beta\in P_+}\ch_{z}L(\beta)\sum_{\sigma\in W}(-1)^{l(\sigma)}\frac{q^{\frac{1}{2}|-\sqrt{t}\sigma(\beta+\rho)+\lambda_t+\frac{1}{\sqrt{t}}\rho|^2}}{\eta(\tau)^{\operatorname{rank}\mathfrak{g}}}.
\end{align}
The singlet $(1,t)$-log VOA is given by $H^0(G\times_{B}V_{\sqrt{t}Q})^{h=0}$ and $H^0(G\times_{B}V_{\sqrt{t}Q+\lambda})^{h=\gamma}$ are its modules.
By using the corollary of the Atiyah--Bott fixed point formula
\begin{align}\label{Atiyah-Bott formula 2}
\ch_{q}H^0(G\times_{B}V)^{h=\gamma}
=\sum_{\beta\in P_+}m_{\beta,\gamma}\sum_{\sigma\in W}(-1)^{l(\sigma)}\ch_{q}V^{h=\sigma\circ\beta},
\end{align}
where $m_{\beta,\gamma}$ is the Kostant multiplicity,
the character of $H^0(G\times_{B}V_{\sqrt{t}Q+\lambda})^{h=\gamma}$ is given by
\begin{align}\label{character of singlet (1,t)-logVOA module}
\ch_{q}H^0(G\times_{B}V_{\sqrt{t}Q+\lambda})^{h=\gamma}
=
\sum_{\beta\in P_+}m_{\beta,\gamma}\sum_{\sigma\in W}(-1)^{l(\sigma)}\frac{q^{\frac{1}{2}|-\sqrt{t}\sigma(\beta+\rho)+\lambda_t+\frac{1}{\sqrt{t}}\rho|^2}}{\eta(\tau)^{\operatorname{rank}\mathfrak{g}}}.
\end{align}

From now on, we consider the case of $\mathfrak{g}=\mathfrak{sl}_2$.
In this case, all $\lambda$ satisfies
$(\sqrt{t}\lambda_0+\rho,\theta)\leq t$, and thus
$H^0(G\times_{B}V_{\sqrt{t}Q+\lambda})$ is irreducible and $H^k(G\times_{B}V_{\sqrt{t}Q+\lambda})=0$ for all $\lambda$ and $k>0$.
To simplify the discussion, we consider the case $V_{n,m}^+=V_{\sqrt{st}Q+\alpha_{n,m}}$ (another case is similar).

Let us check that the character $H^0(G\times_{B}V_{\sqrt{t}Q})$ coincides with the character of $(1,t)$-log VOA.
The irreducible modules of $(1,t)$-log VOA is given by
\begin{align}
H^0(G\times_{B}V_{\sqrt{t}Q+\alpha_{m}})
\quad
(1\leq m\leq t),
\end{align}
where $\alpha_{m}=\frac{m-1}{\sqrt{t}}\varpi$.
The character of the irreducible module $V_{\sqrt{t}Q+\alpha_{m}}$ of the lattice VOA $V_{\sqrt{t}Q}$ is 
\begin{align}
\ch_{q,z}V_{\sqrt{t}Q+\alpha_{m}}
=
\sum_{k\in\mathbb{Z}}\ch_{q,z}\pi_{\alpha_{m}+2k\sqrt{t}\varpi}
=
\sum_{k\in\mathbb{Z}}z^{(\alpha,\alpha_{m}+2k\varpi)}\frac{q^{\Delta_{m,2k}}}{\eta(\tau)},
\end{align}
where $\Delta_{m,k}=\frac{(m-t+kt)^2}{4t}$ is the conformal weight of $e^{\alpha_{m,k}}\in \pi_{\alpha_{m}+k\sqrt{t}\varpi}$ plus $\frac{c}{24}$.
Note that we have
\begin{align}\label{precondition 1,2}
\ch_q(V_{\sqrt{t}Q+\alpha_{m}})^{h=\sigma_1\circ(2k\varpi)}
=
\ch_q(V_{\sqrt{t}(Q-\varpi)+\alpha_{t-m}})^{h=(2k+1)\varpi}
\end{align}
because of $|\beta|^2=|\sigma(\beta)|^2$.
Then we obtain the character of
$H^0(G\times_BV_{\sqrt{t}Q+\alpha_{m}})^{h=0}$ as
\begin{equation}\label{character of irr (1,t)-logVOA}
  \begin{aligned}[b]
&\ch_qH^0(G\times_{B}V_{\sqrt{t}Q+\alpha_{m}})^{h=0}\\
=&\sum_{\beta\in P_+}m_{\beta,0}\sum_{\sigma\in W}(-1)^{l(\sigma)}\ch_q(V_{\sqrt{t}Q+\alpha_{m}})^{h=\sigma\circ\beta}\\
=&\sum_{k\geq 0}\sum_{\sigma\in W}(-1)^{l(\sigma)}\ch_q(V_{\sqrt{t}Q+\alpha_{m}})^{h=\sigma\circ(2k\varpi)}\\
=&\sum_{k\geq 0}
   \left(
   \ch_q(V_{\sqrt{t}Q+\alpha_{m}})^{h=2k\varpi}
   -
   \ch_q(V_{\sqrt{t}(Q-\varpi)+\alpha_{t-m}})^{h=(2k+1)\varpi)}
   \right)\\
  =&\frac{1}{\eta(\tau)}\sum_{k\geq 0}
     \left(
     q^{\Delta_{m,-2k}}-q^{\Delta_{t-m,-2k-1}}
     \right),
  \end{aligned}
\end{equation}
where the first and third equalities follow
from \eqref{Atiyah-Bott formula 2} and \eqref{precondition 1,2} respectively. 
In fact, it coincides with $q^{-\frac{c}{24}}\tilde{\ch}_{W(2,2t-1)}(q)$ in \cite[Section 7]{BringMilas15a} for $m=1$.

\subsection{Geometric construction of $(s,t)$-log VOA and Atiyah--Bott formula}
For the case of $s\geq 2$, the $(s,t)$-log VOA and its irreducible
modules are constructed and studied
algebraically~\cite{FeiGaiSemTip06a, TsuchiyaWood}, but not yet
geometrically.
The second author conjectured that the
irreducible modules $\mathcal{X}_{n,m}^{\pm}$ over $(s,t)$-log VOA is given by studying
``$H^0(G\times_{B}H^0(G\times_{B}V_{n,m}^{\pm}))$".
In the following, we  propose a method to compute the characters of $\mathcal{X}_{n,m}^{\pm}$ in Section \ref{subsec: characters of logVOA}.

We recall some results on $\mathrm{Vir}_{s,t}$ and
$\mathcal{K}_{1,1}^+$ following~\cite{TsuchiyaWood}.
We fix $n$ and $m$ as
$1\leq n< s$, $1\leq m< t$ and $k\in\mathbb{Z}$.
We denote $L_{n,m,k}$ by the irreducible $\operatorname{Vir}_{s,t}$-module with the lowest conformal weight $\Delta_{n,m,k}$~\eqref{define_Delta_n,m,k}.
Hereafter we use
{\tikz[baseline=(T.base)]
  \node[fill=red!50](T){$k$};},
{\tikz[baseline=(T.base)]
  \node[fill=blue!50](T){$k$};},
{\tikz[baseline=(T.base)]
  \node[fill=green!50](T){$k$};},
{\tikz[baseline=(T.base)]]
  \node[fill=yellow!50](T){$k$};}
as the unique simple quotient  given by the irreducible $U(\mathcal{L})$-modules 
$L_{s-n,m,-k}$, $L_{n,m,-k}$, $L_{s-n,t-m,-k}$, $L_{n,t-m,-k}$,
respectively.
Then
it is known~\cite{FeiginFuchs,TsuchiyaWood} (see also \cite{IoharaKoga}) that
the socle sequence of $V_{n,m}^+$ as $\operatorname{Vir}_{s,t}$-module
is given by Fig.~\ref{fig: socle sequence of V+},
and that
the irreducible $\mathcal{K}_{1,1}$-module
$\mathcal{X}_{n,m}^{+}
= \im \mathcal{Q}^{[s-n]}_{+}\cap\im \mathcal{Q}^{[t-m]}_{-}$
is
the $\mathrm{Vir}_{s,t}$-submodule of $V^+_{n,m}$
which consists of all 
{\tikz[baseline=(T.base)]
  {\node[fill=red!50](T){$k$};}
}
in Figure~\ref{fig: socle sequence of V+}.
It was shown~\cite{TsuchiyaWood}
that
$X_{n,m,+}=\im\mathcal{Q}^{[s-n]}_{+}\subseteq V_{n,m}^+$
given in Figure~\ref{fig: socle sequence of X+}
has the $B$-module structure defined by the Frobenius homomorphism
(the $H$-action is given by $h=-\frac{1}{\sqrt{st}}(\alpha_{(0)}-(\alpha,\alpha_{n,m}))$).
Under the $B$-module structure, $\mathcal{X}_{n,m}^+$ is regarded as the maximal $G$-submodule of $X_{n,m,+}$.
In the same manner as the case of $(1,t)$-log VOA \cite[Lemma 4.19]{Sugimoto1}, the map
\begin{align}
H^0(G\times_BX_{n,m,+})\hookrightarrow X_{n,m,+}, \quad 
f\mapsto f(\operatorname{id}_{G/B})
\end{align}
sends $H^0(G\times_BX_{n,m,+})$ to the maximal $G$-submodule of $X_{n,m,+}$.
Therefore we can regard $\mathcal{X}_{n,m}^+\simeq H^0(G\times_{B}X_{n,m,+})$ and set 
\begin{align}\label{align: tildeH=X}
\tilde{H}^0(G\times_BV_{n,m}^{+}):=X_{n,m,+}.
\end{align}

\begin{figure}[htbp]
\centering
\begin{tikzpicture}[scale=.4]
\node[fill=blue!50] (00r) at (10,10) {$0$};
\node[fill=yellow!50] (01l) at (8,8) {$1$};
\node[fill=red!50] (01r) at (10,8) {$1$};
\node[fill=green!50] (02l) at (8,6) {$2$};
\node[fill=blue!50] (02r) at (10,6) {$2$};
\node[fill=yellow!50] (03l) at (8,4) {$3$};
\node[fill=red!50] (03r) at (10,4) {$3$};
\node[fill=green!50] (04l) at (8,2) {$4$};
\node[fill=blue!50] (04r) at (10,2) {$4$};
\node[fill=yellow!50] (05l) at (8,0) {$5$};
\node[fill=red!50] (05r) at (10,0) {$5$};

\draw[arrows=->] (05l)--(04l);
\draw[arrows=->] (03l)--(04l);
\draw[arrows=->] (03l)--(02l);
\draw[arrows=->] (01l)--(02l);
\draw[arrows=->] (01l)--(00r);
\draw[arrows=->] (00r)--(01r);
\draw[arrows=->] (02r)--(01r);
\draw[arrows=->] (02r)--(03r);
\draw[arrows=->] (04r)--(03r);
\draw[arrows=->] (04r)--(05r);

\draw[arrows=->] (05l)--(04r);
\draw[arrows=->] (04l)--(05r);
\draw[arrows=->] (04l)--(03r);
\draw[arrows=->] (03l)--(04r);
\draw[arrows=->] (03l)--(02r);
\draw[arrows=->] (02l)--(03r);
\draw[arrows=->] (02l)--(01r);
\draw[arrows=->] (01l)--(02r);

\node[fill=blue!50] (22r) at (6,6) {$2$};
\node[fill=yellow!50] (23l) at (4,4) {$3$};
\node[fill=red!50] (23r) at (6,4) {$3$};
\node[fill=green!50] (24l) at (4,2) {$4$};
\node[fill=blue!50] (24r) at (6,2) {$4$};
\node[fill=yellow!50] (25l) at (4,0) {$5$};
\node[fill=red!50] (25r) at (6,0) {$5$};

\draw[arrows=->] (25l)--(24l);
\draw[arrows=->] (23l)--(24l);
\draw[arrows=->] (22r)--(23r);
\draw[arrows=->] (24r)--(23r);
\draw[arrows=->] (24r)--(25r);

\draw[arrows=->] (25l)--(24r);
\draw[arrows=->] (24l)--(25r);
\draw[arrows=->] (24l)--(23r);
\draw[arrows=->] (23l)--(24r);
\draw[arrows=->] (23l)--(22r);

\node[fill=blue!50] (44r) at (2,2) {$4$};
\node[fill=yellow!50] (45l) at (0,0) {$5$};
\node[fill=red!50] (45r) at (2,0) {$5$};

\draw[arrows=->] (44r)--(45r);

\draw[arrows=->] (45l)--(44r);

\node[fill=green!50] (-22l) at (12,6) {$2$};
\node[fill=yellow!50] (-23l) at (12,4) {$3$};
\node[fill=red!50] (-23r) at (14,4) {$3$};
\node[fill=green!50] (-24l) at (12,2) {$4$};
\node[fill=blue!50] (-24r) at (14,2) {$4$};
\node[fill=yellow!50] (-25l) at (12,0) {$5$};
\node[fill=red!50] (-25r) at (14,0) {$5$};

\draw[arrows=->] (-25l)--(-24l);
\draw[arrows=->] (-23l)--(-24l);
\draw[arrows=->] (-22l)--(-23r);
\draw[arrows=->] (-24r)--(-23r);
\draw[arrows=->] (-24r)--(-25r);

\draw[arrows=->] (-25l)--(-24r);
\draw[arrows=->] (-24l)--(-25r);
\draw[arrows=->] (-24l)--(-23r);
\draw[arrows=->] (-23l)--(-24r);
\draw[arrows=->] (-23l)--(-22l);

\node[fill=green!50] (-44l) at (16,2) {$4$};
\node[fill=yellow!50] (-45l) at (16,0) {$5$};
\node[fill=red!50] (-45r) at (18,0) {$5$};

\draw[arrows=->] (-44l)--(-45r);

\draw[arrows=->] (-45l)--(-44l);
\end{tikzpicture}
\caption{
 The socle sequence of ${V}_{n,m}^+$ (at $h=4\varpi,2\varpi,0,-2\varpi,-4\varpi$).
 We denote $a\rightarrow b$ as  $b\in U(\mathcal{L})a$.
}
\label{fig: socle sequence of V+}
\end{figure}
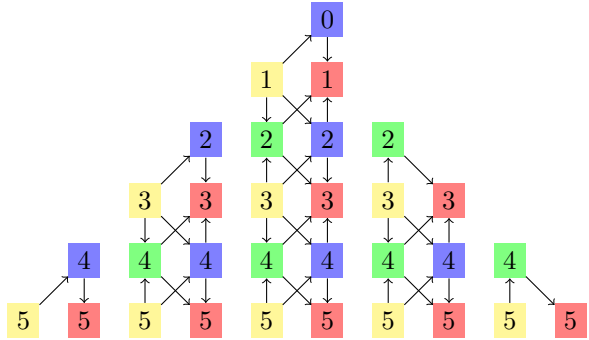

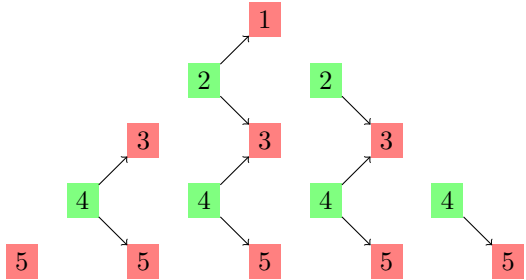
\begin{figure}[htbp]
  \centering
  \begin{tikzpicture}[scale=.4]
    \node[fill=red!50] (01r) at (10,8) {$1$};
    \node[fill=green!50] (02l) at (8,6) {$2$};
    \node[fill=red!50] (03r) at (10,4) {$3$};
    \node[fill=green!50] (04l) at (8,2) {$4$};
    \node[fill=red!50] (05r) at (10,0) {$5$};
    
    \draw[arrows=->] (04l)--(05r);
    \draw[arrows=->] (04l)--(03r);
    \draw[arrows=->] (02l)--(03r);
    \draw[arrows=->] (02l)--(01r);
    
    \node[fill=red!50] (23r) at (6,4) {$3$};
    \node[fill=green!50] (24l) at (4,2) {$4$};
    \node[fill=red!50] (25r) at (6,0) {$5$};
    
    \draw[arrows=->] (24l)--(25r);
    \draw[arrows=->] (24l)--(23r);

    \node[fill=red!50] (45r) at (2,0) {$5$};

    \node[fill=green!50] (-22l) at (12,6) {$2$};
    \node[fill=red!50] (-23r) at (14,4) {$3$};
    \node[fill=green!50] (-24l) at (12,2) {$4$};
    \node[fill=red!50] (-25r) at (14,0) {$5$};

    \draw[arrows=->] (-22l)--(-23r);
    \draw[arrows=->] (-24l)--(-25r);
    \draw[arrows=->] (-24l)--(-23r);

    \node[fill=green!50] (-44l) at (16,2) {$4$};
    \node[fill=red!50] (-45r) at (18,0) {$5$};

    \draw[arrows=->] (-44l)--(-45r);
  \end{tikzpicture}
  \caption{
    The socle sequence of $X_{n,m,+}=\im \mathcal{Q}^{[s-n]}_+$
    is depicted
    at
    $h=4\varpi,2\varpi,0,-2\varpi,-4\varpi$.
    The maximal $G$-submodule
    $\mathcal{X}_{n,m}^+=H^0(G\times_{B}X_{n,m,+})$ consists of
    all
    {\protect\tikz[baseline=(T.base)]
      \protect\node[fill=red!50](T){$k$};
    }
    in $X_{n,m,+}$.}
\label{fig: socle sequence of X+}
\end{figure}

  We introduce the $\operatorname{Vir}_{s,t}$-modules $\tilde{V}_{n,m}^\pm$
  to
  relate $X_{n,m,+}$ with  the form $H^0(G\times_B-)$.
\begin{definition}\label{def: tildeV}
Let $\tilde{V}_{n,m}^+$ and $\tilde{V}_{n,m}^-$ be the
$\operatorname{Vir}_{s,t}$-modules given by the socle sequence in
Fig.~\ref{fig: socle sequence of tildeV+} and
Fig.~\ref{fig: socle sequence of tildeV-},
respectively.
They are $B$-modules
by inclusions and projections.
Precisely
$\tilde{V}_{n,m}^\pm$ are defined as follows.
\begin{itemize}
\item 
$(\tilde{V}_{n,m}^+)^{h=k\varpi\geq 0}$ is the subspace of 
\begin{align}
(\tilde{V}_{n,m}^+)^{h=0}
:=\im Q^{[m]}_{-}\oplus(({V}_{s-n,t-m}^+)^{h=0}/\im Q^{[m]}_{-})
\end{align} 
such that 
$(\tilde{V}_{n,m}^+)^{h=k\varpi\geq 0}
\simeq 
(
\begin{tikzpicture}[baseline=(0.base)]
\node[fill=yellow!50] (0) at (0,0) {$k+1$};
\node[fill=green!50] (1) at (1.4,0) {$k+2$};
\draw[arrows=->] (1)--(0);
\draw[arrows=->] (1)--(2.4,0);
\end{tikzpicture}
\cdots)
\oplus
(
\begin{tikzpicture}[baseline=(0.base)]
\node[fill=blue!50] (0) at (0,0) {$k$};
\node[fill=red!50] (1) at (1.2,0) {$k+1$};
\draw[arrows=->] (1)--(0);
\draw[arrows=->] (1)--(2.2,0);
\end{tikzpicture}
\cdots)
$.
\item 
$(\tilde{V}_{n,m}^+)^{h=-k\varpi<0}$ is the quotient of $(\tilde{V}_{n,m}^+)^{h=0}$ such that
$(\tilde{V}_{n,m}^+)^{h=-k\varpi<0}
\simeq
(
\begin{tikzpicture}[baseline=(0.base)]
\node[fill=green!50] (0) at (0,0) {$k$};
\node[fill=yellow!50] (1) at (1.2,0) {$k+1$};
\draw[arrows=->] (0)--(1);
\draw[arrows=->] (2.2,0)--(1);
\end{tikzpicture}
\cdots)
\oplus
(
\begin{tikzpicture}[baseline=(0.base)]
\node[fill=red!50] (0) at (0,0) {$k-1$};
\node[fill=blue!50] (1) at (1.2,0) {$k$};
\draw[arrows=->] (0)--(1);
\draw[arrows=->] (2,0)--(1);
\end{tikzpicture}
\cdots)
$
.
\item 
$(\tilde{V}_{n,m}^-)^{h=k\varpi>0}$ is the subspace of 
\begin{align}
(\tilde{V}_{n,m}^-)^{h=-\varpi}
:=
\ker Q^{[m]}_{-}\oplus(({V}_{s-n,m}^+)^{h=0}/ \ker Q^{[m]}_{-})
\end{align}
such that
$(\tilde{V}_{n,m}^-)^{h=k\varpi}
\simeq 
(
\begin{tikzpicture}[baseline=(0.base)]
\node[fill=blue!50] (0) at (0,0) {$k$};
\node[fill=red!50] (1) at (1.2,0) {$k+1$};
\draw[arrows=->] (1)--(0);
\draw[arrows=->] (1)--(2.2,0);
\end{tikzpicture}
\cdots)
\oplus
(
\begin{tikzpicture}[baseline=(0.base)]
\node[fill=yellow!50] (0) at (0,0) {$k+1$};
\node[fill=green!50] (1) at (1.4,0) {$k+2$};
\draw[arrows=->] (1)--(0);
\draw[arrows=->] (1)--(2.4,0);
\end{tikzpicture}
\cdots)
$.
\item 
$(\tilde{V}_{n,m}^-)^{h=k\varpi<0}$ is the quotient of $(\tilde{V}_{n,m}^+)^{h=-\varpi}$ such that
$
(\tilde{V}_{n,m}^-)^{h=-k\varpi<0}
\simeq 
(
\begin{tikzpicture}[baseline=(0.base)]
\node[fill=red!50] (0) at (0,0) {$k-1$};
\node[fill=blue!50] (1) at (1.2,0) {$k$};
\draw[arrows=->] (0)--(1);
\draw[arrows=->] (1.9,0)--(1);
\end{tikzpicture}
\cdots)
\oplus
(
\begin{tikzpicture}[baseline=(0.base)]
\node[fill=green!50] (0) at (0,0) {$k$};
\node[fill=yellow!50] (1) at (1.2,0) {$k+1$};
\draw[arrows=->] (0)--(1);
\draw[arrows=->] (2.2,0)--(1);
\end{tikzpicture}
\cdots)
$.
\end{itemize}
\end{definition}

\begin{figure}[htbp]
\centering
\begin{tikzpicture}[scale=.4]
\node[fill=blue!50] (00r) at (10,10) {$0$};
\node[fill=yellow!50] (01l) at (8,8) {$1$};
\node[fill=red!50] (01r) at (10,8) {$1$};
\node[fill=green!50] (02l) at (8,6) {$2$};
\node[fill=blue!50] (02r) at (10,6) {$2$};
\node[fill=yellow!50] (03l) at (8,4) {$3$};
\node[fill=red!50] (03r) at (10,4) {$3$};
\node[fill=green!50] (04l) at (8,2) {$4$};
\node[fill=blue!50] (04r) at (10,2) {$4$};
\node[fill=yellow!50] (05l) at (8,0) {$5$};
\node[fill=red!50] (05r) at (10,0) {$5$};

\draw[arrows=->] (04l)--(05l);
\draw[arrows=->] (04l)--(03l);
\draw[arrows=->] (02l)--(03l);
\draw[arrows=->] (02l)--(01l);
\draw[arrows=->, dotted] (00r)--(01l);
\draw[arrows=->] (01r)--(00r);
\draw[arrows=->] (01r)--(02r);
\draw[arrows=->] (03r)--(02r);
\draw[arrows=->] (03r)--(04r);
\draw[arrows=->] (05r)--(04r);

\draw[arrows=->, dotted] (04r)--(05l);
\draw[arrows=->, dotted] (05r)--(04l);
\draw[arrows=->, dotted] (03r)--(04l);
\draw[arrows=->, dotted] (04r)--(03l);
\draw[arrows=->, dotted] (02r)--(03l);
\draw[arrows=->, dotted] (03r)--(02l);
\draw[arrows=->, dotted] (01r)--(02l);
\draw[arrows=->, dotted] (02r)--(01l);

\node[fill=blue!50] (22r) at (6,6) {$2$};
\node[fill=yellow!50] (23l) at (4,4) {$3$};
\node[fill=red!50] (23r) at (6,4) {$3$};
\node[fill=green!50] (24l) at (4,2) {$4$};
\node[fill=blue!50] (24r) at (6,2) {$4$};
\node[fill=yellow!50] (25l) at (4,0) {$5$};
\node[fill=red!50] (25r) at (6,0) {$5$};

\draw[arrows=->] (24l)--(25l);
\draw[arrows=->] (24l)--(23l);
\draw[arrows=->] (23r)--(22r);
\draw[arrows=->] (23r)--(24r);
\draw[arrows=->] (25r)--(24r);

\draw[arrows=->, dotted] (24r)--(25l);
\draw[arrows=->, dotted] (25r)--(24l);
\draw[arrows=->, dotted] (23r)--(24l);
\draw[arrows=->, dotted] (24r)--(23l);
\draw[arrows=->, dotted] (22r)--(23l);

\node[fill=blue!50] (44r) at (2,2) {$4$};
\node[fill=yellow!50] (45l) at (0,0) {$5$};
\node[fill=red!50] (45r) at (2,0) {$5$};

\draw[arrows=->] (45r)--(44r);

\draw[arrows=->, dotted] (44r)--(45l);

\node[fill=red!50] (-21r) at (14,8) {$1$};
\node[fill=blue!50] (-22r) at (14,6) {$2$};
\node[fill=green!50] (-22l) at (12,6) {$2$};
\node[fill=yellow!50] (-23l) at (12,4) {$3$};
\node[fill=red!50] (-23r) at (14,4) {$3$};
\node[fill=green!50] (-24l) at (12,2) {$4$};
\node[fill=blue!50] (-24r) at (14,2) {$4$};
\node[fill=yellow!50] (-25l) at (12,0) {$5$};
\node[fill=red!50] (-25r) at (14,0) {$5$};

\draw[arrows=->] (-24l)--(-25l);
\draw[arrows=->] (-24l)--(-23l);
\draw[arrows=->] (-22l)--(-23l);
\draw[arrows=->, dotted] (-22r)--(-23l);
\draw[arrows=->] (-23r)--(-24r);
\draw[arrows=->] (-25r)--(-24r);

\draw[arrows=->, dotted] (-24r)--(-25l);
\draw[arrows=->, dotted] (-25r)--(-24l);
\draw[arrows=->, dotted] (-23r)--(-24l);
\draw[arrows=->, dotted] (-24r)--(-23l);
\draw[arrows=->, dotted] (-23r)--(-22l);

\draw[arrows=->] (-21r)--(-22r);
\draw[arrows=->, dotted] (-21r)--(-22l);
\draw[arrows=->] (-23r)--(-22r);
\draw[arrows=->, dotted] (-23r)--(-22l);

\node[fill=red!50] (-43r) at (18,4) {$3$};
\node[fill=blue!50] (-44r) at (18,2) {$4$};
\node[fill=green!50] (-44l) at (16,2) {$4$};
\node[fill=yellow!50] (-45l) at (16,0) {$5$};
\node[fill=red!50] (-45r) at (18,0) {$5$};

\draw[arrows=->, dotted] (-45r)--(-44l);

\draw[arrows=->] (-44l)--(-45l);

\draw[arrows=->, dotted] (-43r)--(-44l);
\draw[arrows=->] (-43r)--(-44r);
\draw[arrows=->] (-45r)--(-44r);
\draw[arrows=->, dotted] (-44r)--(-45l);
\end{tikzpicture}
\caption{
The socle sequence of $\tilde{V}_{n,m}^+$ (at $h=4\varpi,2\varpi,0,-2\varpi,-4\varpi$).
Here
{\protect
\tikz[baseline=(01l.base)]{
  \protect\node (00r) at (0.8,0) {$b$};
  \protect\node (01l) at (0,0) {$a$};
  \protect\draw[arrows=->,dotted] (01l)--(00r);
}}
means that $a\rightarrow b$ in $({V}_{s-n,t-m}^+)^{h=0}$, but not in $\tilde{V}_{n,m}^+$.
}
\label{fig: socle sequence of tildeV+}
\end{figure}
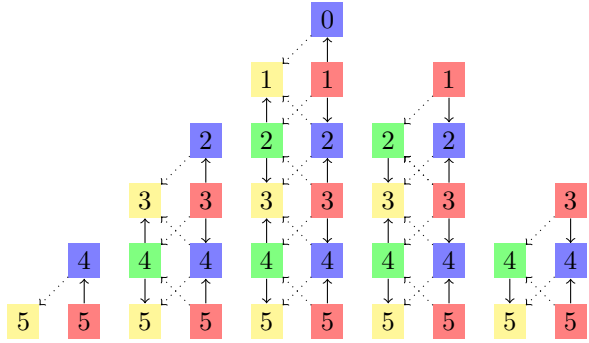

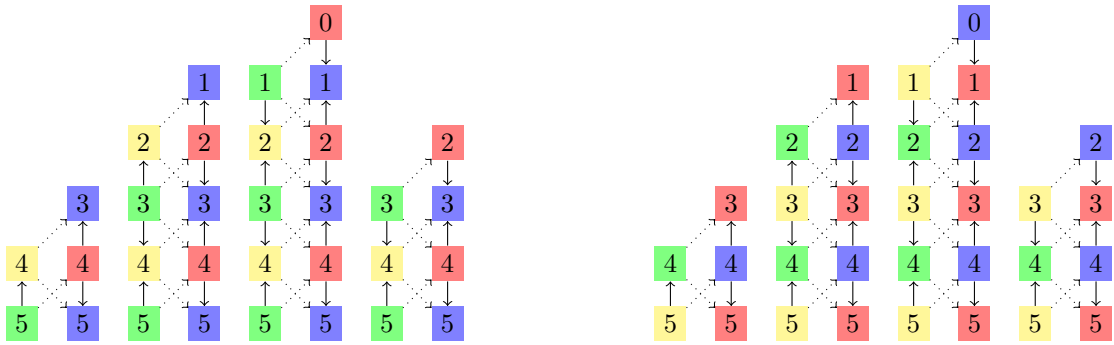
\begin{figure}[htbp]
  \begin{minipage}{.46\linewidth}
\centering
\begin{tikzpicture}[scale=.4]
\node[fill=blue!50] (11r) at (10,8) {$1$};
\node[fill=yellow!50] (12l) at (8,6) {$2$};
\node[fill=red!50] (12r) at (10,6) {$2$};
\node[fill=green!50] (13l) at (8,4) {$3$};
\node[fill=blue!50] (13r) at (10,4) {$3$};
\node[fill=yellow!50] (14l) at (8,2) {$4$};
\node[fill=red!50] (14r) at (10,2) {$4$};
\node[fill=green!50] (15l) at (8,0) {$5$};
\node[fill=blue!50] (15r) at (10,0) {$5$};

;
\draw[arrows=->] (15l)--(14l);
\draw[arrows=->] (13l)--(14l);
\draw[arrows=->] (13l)--(12l);
\draw[arrows=->, dotted] (12l)--(11r);
\draw[arrows=->] (12r)--(11r);
\draw[arrows=->] (12r)--(13r);
\draw[arrows=->] (14r)--(13r);
\draw[arrows=->] (14r)--(15r);

\draw[arrows=->, dotted] (15l)--(14r);
\draw[arrows=->, dotted] (14l)--(15r);
\draw[arrows=->, dotted] (14l)--(13r);
\draw[arrows=->, dotted] (13l)--(14r);
\draw[arrows=->, dotted] (13l)--(12r);
\draw[arrows=->, dotted] (12l)--(13r);

\node[fill=blue!50] (33r) at (6,4) {$3$};
\node[fill=yellow!50] (34l) at (4,2) {$4$};
\node[fill=red!50] (34r) at (6,2) {$4$};
\node[fill=green!50] (35l) at (4,0) {$5$};
\node[fill=blue!50] (35r) at (6,0) {$5$};

\draw[arrows=->] (35l)--(34l);
\draw[arrows=->] (34r)--(33r);
\draw[arrows=->] (34r)--(35r);

\draw[arrows=->, dotted] (35l)--(34r);
\draw[arrows=->, dotted] (34l)--(35r);
\draw[arrows=->, dotted] (34l)--(33r);

\node[fill=red!50] (-10r) at (14,10) {$0$};
\node[fill=blue!50] (-11r) at (14,8) {$1$};
\node[fill=green!50] (-11l) at (12,8) {$1$};
\node[fill=yellow!50] (-12l) at (12,6) {$2$};
\node[fill=red!50] (-12r) at (14,6) {$2$};
\node[fill=green!50] (-13l) at (12,4) {$3$};
\node[fill=blue!50] (-13r) at (14,4) {$3$};
\node[fill=yellow!50] (-14l) at (12,2) {$4$};
\node[fill=red!50] (-14r) at (14,2) {$4$};
\node[fill=green!50] (-15l) at (12,0) {$5$};
\node[fill=blue!50] (-15r) at (14,0) {$5$};

\draw[arrows=->] (-15l)--(-14l);
\draw[arrows=->] (-13l)--(-14l);
\draw[arrows=->] (-13l)--(-12l);
\draw[arrows=->, dotted] (-11l)--(-12r);
\draw[arrows=->] (-12r)--(-13r);
\draw[arrows=->] (-14r)--(-13r);
\draw[arrows=->] (-14r)--(-15r);

\draw[arrows=->, dotted] (-15l)--(-14r);
\draw[arrows=->, dotted] (-14l)--(-15r);
\draw[arrows=->, dotted] (-14l)--(-13r);
\draw[arrows=->, dotted] (-13l)--(-14r);
\draw[arrows=->, dotted] (-13l)--(-12r);
\draw[arrows=->, dotted] (-12l)--(-13r);
\draw[arrows=->] (-11l)--(-12l);

\draw[arrows=->, dotted] (-11l)--(-10r);
\draw[arrows=->] (-10r)--(-11r);
\draw[arrows=->] (-12r)--(-11r);
\draw[arrows=->, dotted] (-12l)--(-11r);

\node[fill=red!50] (-32r) at (18,6) {$2$};
\node[fill=blue!50] (-33r) at (18,4) {$3$};
\node[fill=green!50] (-33l) at (16,4) {$3$};
\node[fill=yellow!50] (-34l) at (16,2) {$4$};
\node[fill=red!50] (-34r) at (18,2) {$4$};
\node[fill=green!50] (-35l) at (16,0) {$5$};
\node[fill=blue!50] (-35r) at (18,0) {$5$};

\draw[arrows=->, dotted] (-34l)--(-35r);
\draw[arrows=->, dotted] (-33l)--(-34r);

\draw[arrows=->] (-35l)--(-34l);
\draw[arrows=->] (-33l)--(-34l);
\draw[arrows=->] (-34r)--(-35r);
\draw[arrows=->, dotted] (-35l)--(-34r);

\draw[arrows=->, dotted] (-33l)--(-32r);
\draw[arrows=->] (-32r)--(-33r);
\draw[arrows=->] (-34r)--(-33r);
\draw[arrows=->, dotted] (-34l)--(-33r);
\end{tikzpicture}
\end{minipage}
\hfill
\begin{minipage}{.46\linewidth}
  \centering
\begin{tikzpicture}[scale=.4]
\node[fill=red!50] (11r) at (10,8) {$1$};
\node[fill=green!50] (12l) at (8,6) {$2$};
\node[fill=blue!50] (12r) at (10,6) {$2$};
\node[fill=yellow!50] (13l) at (8,4) {$3$};
\node[fill=red!50] (13r) at (10,4) {$3$};
\node[fill=green!50] (14l) at (8,2) {$4$};
\node[fill=blue!50] (14r) at (10,2) {$4$};
\node[fill=yellow!50] (15l) at (8,0) {$5$};
\node[fill=red!50] (15r) at (10,0) {$5$};

;
\draw[arrows=->] (15l)--(14l);
\draw[arrows=->] (13l)--(14l);
\draw[arrows=->] (13l)--(12l);
\draw[arrows=->, dotted] (12l)--(11r);
\draw[arrows=->] (12r)--(11r);
\draw[arrows=->] (12r)--(13r);
\draw[arrows=->] (14r)--(13r);
\draw[arrows=->] (14r)--(15r);

\draw[arrows=->, dotted] (15l)--(14r);
\draw[arrows=->, dotted] (14l)--(15r);
\draw[arrows=->, dotted] (14l)--(13r);
\draw[arrows=->, dotted] (13l)--(14r);
\draw[arrows=->, dotted] (13l)--(12r);
\draw[arrows=->, dotted] (12l)--(13r);

\node[fill=red!50] (33r) at (6,4) {$3$};
\node[fill=green!50] (34l) at (4,2) {$4$};
\node[fill=blue!50] (34r) at (6,2) {$4$};
\node[fill=yellow!50] (35l) at (4,0) {$5$};
\node[fill=red!50] (35r) at (6,0) {$5$};

\draw[arrows=->] (35l)--(34l);
\draw[arrows=->] (34r)--(33r);
\draw[arrows=->] (34r)--(35r);

\draw[arrows=->, dotted] (35l)--(34r);
\draw[arrows=->, dotted] (34l)--(35r);
\draw[arrows=->, dotted] (34l)--(33r);

\node[fill=blue!50] (-10r) at (14,10) {$0$};
\node[fill=red!50] (-11r) at (14,8) {$1$};
\node[fill=yellow!50] (-11l) at (12,8) {$1$};
\node[fill=green!50] (-12l) at (12,6) {$2$};
\node[fill=blue!50] (-12r) at (14,6) {$2$};
\node[fill=yellow!50] (-13l) at (12,4) {$3$};
\node[fill=red!50] (-13r) at (14,4) {$3$};
\node[fill=green!50] (-14l) at (12,2) {$4$};
\node[fill=blue!50] (-14r) at (14,2) {$4$};
\node[fill=yellow!50] (-15l) at (12,0) {$5$};
\node[fill=red!50] (-15r) at (14,0) {$5$};

\draw[arrows=->] (-15l)--(-14l);
\draw[arrows=->] (-13l)--(-14l);
\draw[arrows=->] (-13l)--(-12l);
\draw[arrows=->, dotted] (-11l)--(-12r);
\draw[arrows=->] (-12r)--(-13r);
\draw[arrows=->] (-14r)--(-13r);
\draw[arrows=->] (-14r)--(-15r);

\draw[arrows=->, dotted] (-15l)--(-14r);
\draw[arrows=->, dotted] (-14l)--(-15r);
\draw[arrows=->, dotted] (-14l)--(-13r);
\draw[arrows=->, dotted] (-13l)--(-14r);
\draw[arrows=->, dotted] (-13l)--(-12r);
\draw[arrows=->, dotted] (-12l)--(-13r);
\draw[arrows=->] (-11l)--(-12l);

\draw[arrows=->, dotted] (-11l)--(-10r);
\draw[arrows=->] (-10r)--(-11r);
\draw[arrows=->] (-12r)--(-11r);
\draw[arrows=->, dotted] (-12l)--(-11r);

\node[fill=blue!50] (-32r) at (18,6) {$2$};
\node[fill=red!50] (-33r) at (18,4) {$3$};
\node[fill=yellow!50] (-33l) at (16,4) {$3$};
\node[fill=green!50] (-34l) at (16,2) {$4$};
\node[fill=blue!50] (-34r) at (18,2) {$4$};
\node[fill=yellow!50] (-35l) at (16,0) {$5$};
\node[fill=red!50] (-35r) at (18,0) {$5$};

\draw[arrows=->, dotted] (-34l)--(-35r);
\draw[arrows=->, dotted] (-33l)--(-34r);

\draw[arrows=->] (-35l)--(-34l);
\draw[arrows=->] (-33l)--(-34l);
\draw[arrows=->] (-34r)--(-35r);
\draw[arrows=->, dotted] (-35l)--(-34r);

\draw[arrows=->, dotted] (-33l)--(-32r);
\draw[arrows=->] (-32r)--(-33r);
\draw[arrows=->] (-34r)--(-33r);
\draw[arrows=->, dotted] (-34l)--(-33r);
\end{tikzpicture}
\end{minipage}
\caption{The socle sequences of $\tilde{V}_{n,m}^-$
  (left)
  and
  $\tilde{V}_{s-n,m}^-$
  (right)
  are depicted 
  at $h=3\varpi,\varpi,-\varpi,-3\varpi$.
}
\label{fig: socle sequence of tildeV-}
\end{figure}

The key observation is that $\tilde{V}_{n,m}^+$ in
Fig.~\ref{fig: socle sequence of tildeV+}
and $X_{n,m,+}$ in
Fig.~\ref{fig: socle sequence of X+}
have the same shape if
we regard
$(
\begin{tikzpicture}[baseline=(00r.base)]
\node[fill=green!50] (00r) at (4/3,0) {$k+1$};
\node[fill=red!50] (01l) at (0,0) {$k$};
\draw[arrows=->,dotted] (01l)--(00r);
\end{tikzpicture}
)$
and 
$(
\begin{tikzpicture}[baseline=(00r.base)]
\node[fill=yellow!50] (00r) at (4/3,0) {$k+1$};
\node[fill=blue!50] (01l) at (0,0) {$k$};
\draw[arrows=->,dotted] (01l)--(00r);
\end{tikzpicture}
)$
in Fig.~\ref{fig: socle sequence of tildeV+}
as one component each.
This is intended to treat $\tilde{V}_{n,m}^+$ and $X_{n,m,+}$ (or $V_{\sqrt{t}Q+\alpha_m}$ in the previous subsection appearing in the $(1,t)$-log VOA setting) as if they were the same.
The following proposition is essential for the computation of the
character of $\mathcal{X}_{r,s}^{\pm}$.
\begin{prop}\label{prop: conditions for Atiyah-Bott method}
  The $B$-modules
  $\tilde{H}^0(G\times_{B}{V}_{n,m}^+)$~\eqref{align: tildeH=X}
  and $\tilde{V}_{n,m}^+$ in
Def.~\ref{def: tildeV} satisfy the following
(the same results holds if $Q$ is changed to $Q-\varpi$).
\begin{enumerate}
\item
For $\beta\in\mathbb{Z}_{\geq 0}\varpi$, we have
\begin{align}
\label{condition 1}
\ch_q\tilde{H}^0(G\times_{B}{V}_{n,m}^+)^{h=\sigma_1\circ\beta}
&=
\ch_q{H}^0(G\times_{B}\tilde{V}_{n,t-m}^-)^{h=\beta+\varpi},\\
\label{condition 2}
\ch_q(\tilde{V}_{n,m}^+)^{h=\sigma_1\circ\beta}
&=
\ch_q(\tilde{V}_{s-n,m}^-)^{h=\beta+\varpi}.
\end{align}
\item
For $k> 0$ and $\beta\in\mathbb{Z}_{\geq 0}\varpi$, we have
\begin{align}
\label{condition 3}
\ch_q\tilde{H}^0(G\times_{B}{V}_{n,m}^+)^{h=\beta}
&=
\ch_q{H}^0(G\times_{B}\tilde{V}_{s-n,m}^-)^{h=\beta+\varpi},\\
\label{condition 5}
\ch_q(\tilde{V}_{n,m}^+)^{h=k\varpi}
&=
\ch_q({V}_{n,m}^+)^{h=k\varpi}.
\end{align}
\item
We have the cohomology vanishings
\begin{align}\label{align: cohomology vanishing condition}
H^1(G\times_{B}\tilde{H}^0(G\times_{B}{V}_{n,m}^+))
={H}^1(G\times_{B}\tilde{V}_{n,m}^+)
=0.
\end{align}
\end{enumerate} 
\end{prop}

\begin{proof}
  \begin{enumerate}
  \item 
  By comparing
  Fig.~\ref{fig: socle sequence of tildeV+} with
  Fig.~\ref{fig: socle sequence of tildeV-}, we obtain \eqref{condition 2}.
  In the same manner, \eqref{condition 1} is also proved.

  \item
    By Fig.~\ref{fig: socle sequence of V+} and
    Fig.~\ref{fig: socle sequence of tildeV+}, we have \eqref{condition 5}.
    The socle sequence of $H^0(G\times_B\tilde{V}_{s-n,m}^-)$
    (namely, maximal $G$-submodule of $\tilde{V}_{s-n,m}^-$) is given by
    Fig.~\ref{fig: socle sequence of H(tildeV-)}.
    Then
    we obtain~\eqref{condition 3}
    by comparing with
    Fig.~\ref{fig: socle sequence of X+}.

  \item
The cohomology vanishing \eqref{align: cohomology vanishing condition}
follows in the same manner as the case of $(1,t)$-log VOA \cite[Lemma
4.10]{Sugimoto1}.
\end{enumerate}
\end{proof}

\begin{figure}[tbh]
\centering
\begin{tikzpicture}[scale=.4]
\node[fill=red!50] (11r) at (10,8) {$1$};
\node[fill=green!50] (12l) at (8,6) {$2$};
\node[fill=red!50] (13r) at (10,4) {$3$};
\node[fill=green!50] (14l) at (8,2) {$4$};
\node[fill=red!50] (15r) at (10,0) {$5$};

\draw[arrows=->, dotted] (12l)--(11r);
\draw[arrows=->, dotted] (14l)--(15r);
\draw[arrows=->, dotted] (14l)--(13r);
\draw[arrows=->, dotted] (12l)--(13r);

\node[fill=red!50] (33r) at (6,4) {$3$};
\node[fill=green!50] (34l) at (4,2) {$4$};
\node[fill=red!50] (35r) at (6,0) {$5$};

\draw[arrows=->, dotted] (34l)--(35r);
\draw[arrows=->, dotted] (34l)--(33r);

\node[fill=red!50] (-11r) at (14,8) {$1$};
\node[fill=green!50] (-12l) at (12,6) {$2$};
\node[fill=red!50] (-13r) at (14,4) {$3$};
\node[fill=green!50] (-14l) at (12,2) {$4$};
\node[fill=red!50] (-15r) at (14,0) {$5$};

\draw[arrows=->, dotted] (-14l)--(-15r);
\draw[arrows=->, dotted] (-14l)--(-13r);
\draw[arrows=->, dotted] (-12l)--(-13r);
\draw[arrows=->, dotted] (-12l)--(-11r);

\node[fill=red!50] (-33r) at (18,4) {$3$};
\node[fill=green!50] (-34l) at (16,2) {$4$};
\node[fill=red!50] (-35r) at (18,0) {$5$};

\draw[arrows=->, dotted] (-34l)--(-35r);
\draw[arrows=->, dotted] (-34l)--(-33r);
\end{tikzpicture}
\caption{The socle sequence of $H^0(G\times_B\tilde{V}_{s-n,m}^-)$
  is given at $h=3\varpi,\varpi,-\varpi,-3\varpi$.
}
\label{fig: socle sequence of H(tildeV-)}
\end{figure}
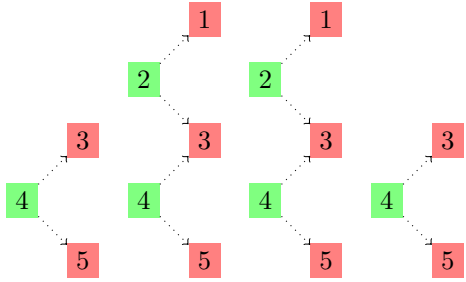

\begin{theorem}
  \begin{equation}
    \ch_q (\mathcal{X}_{n,m}^\pm)^{h=0}
    =
    H^0(G\times_B \tilde{H}^0(G\times_B
    {V}_{n,m}^\pm))^{h=0} .
    \label{character of irr (s,t)-logVOA}  
  \end{equation}
\end{theorem}
\begin{proof}
  As a consequence of
  Prop.~\ref{prop: conditions for Atiyah-Bott method},
we can apply  the Atiyah--Bott formula~\eqref{Atiyah-Bott formula 2}.
The proof is as follows.
\begin{align*}
  &\ch_qH^0(G\times_B \tilde{H}^0(G\times_B{V}_{n,m}^+))^{h=0}
  \\
  \stackrel{
  \eqref{Atiyah-Bott formula 2}
  }{=}
  &
    \sum_{\beta\in P_+}m_{\beta,0}\sum_{\sigma\in W}(-1)^{l(\sigma)}
    \ch_q\tilde{H}^0(G\times_B{V}_{n,m}^+)^{h=\sigma\circ\beta}
  \\
  \stackrel{
  \hphantom{AAA}
  }{=}
  &\sum_{k\geq 0}\sum_{\sigma\in W}(-1)^{l(\sigma)}
    \ch_q\tilde{H}^0(G\times_B{V}_{n,m}^+)^{h=\sigma\circ
    2k\varpi}
  \\
  \stackrel{
  \eqref{condition 1}
  }{=}
  &\sum_{k\geq 0}
    \left(
    \ch_q\tilde{H}^0(G\times_B{V}_{n,m}^+)^{h= 2k\varpi}
    -
    \ch_q\tilde{H}^0(G\times_B{V}_{n,t-m}^-)^{h=
    (2k+1)\varpi}
    \right)
  \\
  \stackrel{
  \eqref{condition 3}}{=}
  &\sum_{k\geq 0}
    \left(
    \ch_q{H}^0(G\times_B\tilde{V}_{s-n,m}^-)^{h= (2k+1)\varpi}
    -
    \ch_q{H}^0(G\times_B\tilde{V}_{s-n,t-m}^+)^{h=
    (2k+2)\varpi}
    \right)
  \\
  \stackrel{
  \eqref{Atiyah-Bott formula 2}
  }{=}
  &\sum_{k\geq 0}
    \left(
    \sum_{\beta\in P_+}m_{\beta,(2k+1)\varpi}\sum_{\sigma\in W}(-1)^{l(\sigma)}
    \ch_q(\tilde{V}_{s-n,m}^-)^{h=\sigma\circ\beta}
    \right.
    \\
    &\qquad
    \left.
      -\sum_{\beta\in P_+}m_{\beta,(2k+2)\varpi}\sum_{\sigma\in W}(-1)^{l(\sigma)}
      \ch_q(\tilde{V}_{s-n,t-m}^+)^{h=\sigma\circ\beta}
      \right)
  \\
  \stackrel{
  \hphantom{AAA}
  }{=}
  &\sum_{k\geq 0}
    \left(
    \sum_{k'\geq 0}\sum_{\sigma\in W}(-1)^{l(\sigma)}
    \ch_q(\tilde{V}_{s-n,m}^-)^{h=\sigma\circ(2k+2k'+1)\varpi}
    -\sum_{k'\geq 0}\sum_{\sigma\in W}(-1)^{l(\sigma)}
    \ch_q(\tilde{V}_{s-n,t-m}^+)^{h=\sigma\circ(2k+2k'+2)\varpi}
    \right)
    \\
    \stackrel{
    \eqref{condition 2}}{=}
    &\sum_{k,k'\geq 0}
       \left(
       \left(
   \ch_q(\tilde{V}_{s-n,m}^-)^{h=(2k+2k'+1)\varpi}
-
\ch_q(\tilde{V}_{n,m}^+)^{h=(2k+2k'+2)\varpi}
       \right)
      \right.\\
    &-
  \left.
  \left(
  \ch_q(\tilde{V}_{s-n,t-m}^+)^{h=(2k+2k'+2)\varpi}
  -
  \ch_q(\tilde{V}_{n,t-m}^-)^{h=(2k+2k'+3)\varpi}
  \right)
  \right)\\
    \stackrel{\eqref{condition 5}
    }{=}
    &\sum_{k,k'\geq 0}
      \left(
  \left(
  \ch_q(\tilde{V}_{s-n,m}^-)^{h=(2k+2k'+1)\varpi}
  -
  \ch_q(\tilde{V}_{n,m}^+)^{h=(2k+2k'+2)\varpi}
  \right)
  \right.
\\
&\left.
  -\left(
  \ch_q(\tilde{V}_{s-n,t-m}^+)^{h=(2k+2k'+2)\varpi}
  -
  \ch_q(\tilde{V}_{n,t-m}^-)^{h=(2k+2k'+3)\varpi}
  \right)
  \right)\\
    \stackrel{
    \hphantom{AAA}
    }{=}
&\frac{1}{\eta(\tau)}\sum_{k,k'\geq 0}
   \left(
   q^{\Delta_{s-n,m,-2k-2k'-1}}-q^{\Delta_{s-n,t-m,-2k-2k'-2}}-q^{\Delta_{n,m,-2k-2k'-2}}+q^{\Delta_{n,t-m,-2k-2k'-3}}
  \right).
\end{align*}
By~\eqref{relation among conformal weight},
this coincides with
$\ch_q (\mathcal{X}_{n,m}^+)^{h=0}$
(see~\cite[(3.43),(3.52)]{ChChFeFeGuHaPa22a}).

We can compute the character $\ch_q (\mathcal{X}_{n,m}^-)^{h=0}$  in the same manner.
\end{proof}

\begin{remark}
    The characters~\eqref{character_X+}
    and~\eqref{character_X-}
    are given by
  \begin{equation}
    \label{eq:4}
    \ch_q(\mathcal{X}_{n,m}^\pm)
    =
    \ch_{q}H^0(G\times_B\tilde{H}^0(G\times_B{V}_{n,m}^\pm)) ,
  \end{equation}
  which follow from
  \begin{equation}
    \begin{aligned}[b]
      &\ch_{q,z}H^0(G\times_B\tilde{H}^0(G\times_B{V}_{n,m}^+))\\
      =&\frac{1}{\eta(\tau)}\sum_{k,k'\geq 0}\ch_zL(2k)
         \left(
         q^{\Delta_{s-n,m,-2k-2k'-1}}-q^{\Delta_{s-n,t-m,-2k-2k'-2}}
         -q^{\Delta_{n,m,-2k-2k'-2}}+q^{\Delta_{n,t-m,-2k-2k'-3}}
         \right) .
    \end{aligned}
  \end{equation}

\end{remark}
We have obtained the character by introducing  $(\tilde{V}_{n,m}^\pm)^{h=k\varpi}$ in Definition \ref{def: tildeV}.
Another
method
would be
to disregard
the commutativity of the $B$-action with the Virasoro action except for the conformal grading instead of taking $(\tilde{V}_{n,m}^\pm)^{h=k\varpi}$ as a Fock space.
There remains for a future work to 
introduce screening operators that define such $B$-action.
We hope to report on
a geometrical
construction of the characters of $(s,t)$-log VOA for $\mathfrak{g}$,
and on
a relationship with
$\mathfrak{g}$-quantum invariant for torus link $T_{r s, r t}$
as a generalization
of~\cite{Kanade23b}.
%

\section*{Acknowledgments}
The authors would like to thank Shashank Kanade and
Toshiki Matsusaka for useful communications.
The work of KH is supported in part by
JSPS KAKENHI Grant Numbers
JP22H01117,
JP20K03601,
JP20K03931,
JP16H03927.
SS is supported by
JSPS KAKENHI Grant Number 22J00951.


\begin{thebibliography}{10}
\providecommand{\url}[1]{\texttt{#1}}
\providecommand{\urlprefix}{URL }
\providecommand{\eprint}[2][]{\url{#2}}

\bibitem{AtiyahBott} 
  M.~F.~Atiyah
  and R.~Bott,
  \emph{A Lefschetz fixed point formula for elliptic complexes I, II},
  \href{https://annals.math.princeton.edu/1967/86-2/p07}
{Ann. of Math.}, \textbf{86}, 374--407 (1967);
\textbf{88}, 451--491  (1968).

\bibitem{BringMilas15a}
K.~Bringmann and A.~Milas, \emph{{$\mathcal{W}$}-algebras, false theta
  functions and quantum modular forms, {I}},
  \href{http://dx.doi.org/10.1093/imrn/rnv033}{Int. Math. Res. Not. IMRN}
  \textbf{2015}, 11351--11387 (2015).

\bibitem{ChChFeFeGuHaPa22a}
M.~C.~N. Cheng, S.~Chun, B.~Feigin, F.~Ferrari, S.~Gukov, S.~M. Harrison, and
  D.~Passaro, \emph{3-manifolds and {VOA} characters}, preprint  (2022),
  \href{http://arxiv.org/abs/2201.04640}{\texttt{arXiv:2201.04640 [hep-th]}}.

\bibitem{DasbaXSLin06a}
O.~T. Dasbach and X.-S. Lin, \emph{On the head and the tail of the colored
  {Jones} polynomial},
  \href{http://dx.doi.org/10.1112/S0010437X06002296}{Compositio Math.}
  \textbf{142}, 1332--1342 (2006).

\bibitem{FeiginFuchs}
 B.~L. Feigin and D.B.~Fuchs, 
 \emph{Representations of the Virasoro algebra},
 \href{https://zbmath.org/?q=an:0722.17020}{Representations of Lie Groups and Related Topics, Adv. Stud. Contemp. Math., Gordon and Breach, New York}, 
 \textbf{7} 465–554 (1990).

\bibitem{FeiGaiSemTip06a}
B.~L. Feigin, A.~M. Gainutdinov, A.~M. Semikhatov, and I.~{\relax Yu}. Tipunin,
  \emph{Logarithmic extensions of minimal models: characters and modular
  transformations},
  \href{http://dx.doi.org/10.1016/j.nuclphysb.2006.09.019}{Nucl. Phys. B}
  \textbf{757}, 303--343 (2006).

\bibitem{FeiTip10}
B.~L.~Feigin, I.~Yu.~Tipunin. \emph{Logarithmic CFTs connected with simple Lie algebras}, 
  \href{https://arxiv.org/abs/1002.5047}{arXiv:1002.5047 [math.QA]}.

\bibitem{FujiIwakMuraTera08a}
H.~Fuji, K.~Iwaki, H.~Murakami, and Y.~Terashima,
  \emph{{Witten--Reshetikhin--Turaev} function for a knot in {Seifert}
  manifolds}, \href{http://dx.doi.org/10.1007/s00220-021-03953-y}{Commun. Math.
  Phys.} \textbf{386}, 225--251 (2021).

\bibitem{KHikami03a}
K.~Hikami, \emph{Quantum invariant for torus link and modular forms},
  \href{http://dx.doi.org/10.1007/s00220-004-1046-2}{Commun. Math. Phys.}
  \textbf{246}, 403--426 (2004).

\bibitem{KHikami04b}
  K. Hikami, \emph{On the quantum invariant for the {Brieskorn} homology
  spheres}, \href{http://dx.doi.org/10.1142/S0129167X05003004}{Int. J. Math.}
  \textbf{16}, 661--685 (2005).

\bibitem{KHikami04e}
  K. Hikami, \emph{Quantum invariant, modular form, and lattice points},
  \href{http://dx.doi.org/10.1155/IMRN.2005.121}{Int. Math. Res. Not. IMRN}
  \textbf{2005}, 121--154 (2005).

\bibitem{KHikami05a}
  K. Hikami,
  \emph{On the quantum invariant for the spherical {Seifert}
  manifold}, \href{http://dx.doi.org/10.1007/s00220-006-0094-1}{Commun. Math.
  Phys.} \textbf{268}, 285--319 (2006).

\bibitem{KHikami10a}
  K. Hikami,
  \emph{Decomposition of the {Witten--Reshetikhin--Turaev}
  invariants: linking pairing and modular forms}, in J.~E. Andersen, H.~U.
  Boden, A.~Hahn, and B.~Himpel, eds., \emph{{Chern--Simons} Gauge Theory: 20
  Years After}, vol.~50 of \emph{AMS/IP Studies in Advanced Mathematics}, pp.
  131--151, Amer. Math. Soc., Providence, 2011.

\bibitem{KHikami03c}
K.~Hikami and A.~N. Kirillov, \emph{Torus knot and minimal model},
  \href{http://dx.doi.org/10.1016/j.physletb.2003.09.007}{Phys. Lett. B}
  \textbf{575}, 343--348 (2003).

\bibitem{IoharaKoga}
K.~Iohara and Y.~Koga, 
\emph{Representation Theory of the Virasoro Algebra},
\href{https://link.springer.com/book/10.1007/978-0-85729-160-8}{Springer},
2010.


\bibitem{Kanade23a}
S.~Kanade, \emph{Coloured {$\mathfrak{sl}_r$} invariants of torus knots and
  characters of {$\mathcal{W}_r$} algebras},
  \href{http://dx.doi.org/10.1007/s11005-022-01628-w}{Lett. Math. Phys.}
  \textbf{113}, 5 (2023), 21 pages.

\bibitem{Kanade23b}
  S.~Kanade,
  \emph{Characters of logarithmic vertex operator algebras and
    coloured invariants of torus links},
    \href{https://arxiv.org/abs/2305.17543}{arXiv:2305.17543}.
  
\bibitem{Kasha95}
R.~M. Kashaev, \emph{A link invariant from quantum dilogarithm},
  \href{http://dx.doi.org/10.1142/S0217732395001526}{Mod. Phys. Lett. A}
  \textbf{10}, 1409--1418 (1995).

\bibitem{LawrZagi99a}
R.~Lawrence and D.~Zagier, \emph{Modular forms and quantum invariants of
  3-manifolds}, Asian J. Math. \textbf{3}, 93--107 (1999).

\bibitem{LinZhe10a}
  X.-S. Lin and H. Zheng,
  \emph{On the Hecke algebras and the colored HOMFLY polynomial},
  \href{http://dx.doi.org/10.1090/S0002-9947-09-04691-1}{Trans. Amer. Math. Soc.}
  \textbf{361}, 1--18 (2010).
  
\bibitem{MatsuTeras21a}
T.~Matsusaka and Y.~Terashima, \emph{Modular transformations of homological
  blocks for {Seifert} fibered homology 3-spheres},
  \href{http://arxiv.org/abs/2112.06210}{\texttt{arXiv:2112.06210 [math.GT]}}.

\bibitem{Mort95a}
H.~R. Morton, \emph{The coloured {Jones} function and {Alexander} polynomial
  for torus knots}, \href{http://dx.doi.org/10.1017/S0305004100072959}{Proc.
  Cambridge Philos. Soc.} \textbf{117}, 129--135 (1995).

\bibitem{MuraMura99a}
H.~Murakami and J.~Murakami, \emph{The colored {Jones} polynomials and the
  simplicial volume of a knot},
  \href{http://dx.doi.org/10.1007/BF02392716}{Acta Math.} \textbf{186}, 85--104
  (2001).

\bibitem{RossJone93a}
M.~Rosso and V.~Jones, \emph{On the invariants of torus knots derived from
  quantum groups}, \href{http://dx.doi.org/10.1142/S0218216593000064}{J. Knot
  Theory Ramifications} \textbf{2}, 97--112 (1993).

\bibitem{Sugimoto1} S.~Sugimoto, \emph{On the Feigin--Tipunin
    conjecture},
  \href{https://link.springer.com/article/10.1007/s00029-021-00662-1#citeas}{Selecta
    Math.} \textbf{27}, 86 (2021).



\bibitem{Sugimoto2} S. ~Sugimoto, \emph{Simplicity of higher rank
    triplet $W$-algebras},
  \href{https://academic.oup.com/imrn/advance-article-abstract/doi/10.1093/imrn/rnac189/6650828}{Int. Math. Res. Not. IMRN},
  rnac189, (2022).

\bibitem{TsuchiyaWood} A.~Tsuchiya and S.~Wood, \emph{On the extended
    $W$-algebra of type $\mathfrak{sl}_2$ at positive rational level},
  \href{https://academic.oup.com/imrn/article-abstract/2015/14/5357/775903?redirectedFrom=fulltext}{Int. Math. Res. Not. IMRN},
  \textbf{2015},
  5357--5435 (2015).


\bibitem{DZagie01a}
D.~Zagier, \emph{Vassiliev invariants and a strange identity related to the
  {Dedekind} eta-function},
  \href{http://dx.doi.org/10.1016/S0040-9383(00)00005-7}{Topology} \textbf{40},
  945--960 (2001).

\bibitem{Zagier09a}
  D. Zagier,
  \emph{Quantum modular forms}, in E.~Blanchard, D.~Ellwood,
  M.~Khalkhali, M.~Marcolli, H.~Moscovici, and S.~Popa, eds., \emph{Quanta of
  Maths}, vol.~11 of \emph{Clay Mathematics Proceedings}, pp. 659--675, Amer.
  Math. Soc., Providence, 2010.

\end{thebibliography}

\end{document}